\documentclass[a4paper]{article}

\usepackage[a4paper,width=150mm,top=25mm,bottom=35mm]{geometry}

\usepackage{amsmath}
\usepackage{amssymb}
\usepackage{amsthm}
\usepackage{bbm}

\usepackage{enumerate}

\numberwithin{equation}{section}

\author{Edgar Assing}
\title{On sup-norm bounds part I: ramified Maa\ss\  newforms over number fields}

\theoremstyle{plain}
\newtheorem{lemma}{Lemma}[section]
\newtheorem{prop}{Proposition}[section]
\newtheorem{theorem}{Theorem}[section]
\newtheorem{cor}{Corollary}[section]
\newtheorem{rem}{Remark}[section]

\newtheorem{conv}{Convention}

\DeclareMathOperator{\R}{\mathbb{R}}
\DeclareMathOperator{\C}{\mathbb{C}}
\DeclareMathOperator{\A}{\mathbb{A}}
\DeclareMathOperator{\Afin}{\mathbb{A}_{\text{fin}}}
\DeclareMathOperator{\Q}{\mathbb{Q}}
\DeclareMathOperator{\Z}{\mathbb{Z}}
\DeclareMathOperator{\N}{\mathbb{N}}
\DeclareMathOperator{\No}{\mathcal{N}}

\DeclareMathOperator{\n}{\mathfrak{n}}
\DeclareMathOperator{\p}{\mathfrak{p}}
\DeclareMathOperator{\Of}{\mathcal{O}_F}
\DeclareMathOperator{\op}{\mathfrak{o}}

\DeclareMathOperator{\vol}{\text{Vol}}

\DeclareMathOperator{\sgn}{\text{sgn}}

\DeclareMathOperator{\into}{\hookrightarrow}

\newcommand{\abs}[1]{\left\vert #1 \right\vert}
\newcommand{\du}[1]{\underline{\underline{#1}}}
\newcommand{\nilp}[1]{\left(\begin{matrix} 1&#1 \\ 0&1\end{matrix}\right)}
\newcommand{\cent}[1]{\left(\begin{matrix} #1&0 \\ 0&#1\end{matrix}\right)}
\newcommand{\am}[1]{\left(\begin{matrix} #1&0 \\ 0&1\end{matrix}\right)}
\newcommand{\SP}[2]{\left\langle #1, #2 \right\rangle}

\begin{document}

\maketitle

\begin{abstract}
We prove new upper bounds for the sup-norm of Hecke Maa\ss\  newforms on $GL(2)$ over a number field. Our newforms are more general than those considered in a recent paper by Blomer, Harcos, Maga, and Mili\`cevi\`c: we do not require square free level. Furthermore, we allow for non-trivial central character. Over the rationals we cover the best bounds as obtained by Saha.
\end{abstract}

\tableofcontents

\section{Introduction}

In this paper we prove bounds for the $L^{\infty}$-norm of automorphic forms on $GL_2$ which improve upon the local bounds. The problem of establishing such bounds is commonly referred to as sup-norm problem for $GL_2$. Just recently in the paper \cite{BHMM16} this problem was solved over number fields for square free level and trivial central character. Previously, in \cite{Sa15}, this problem was solved over $\Q$ with arbitrary level and central character. In the present paper we go beyond these ideas. Containing each of them as special cases, our results recover the strongest bounds from both.

The true size of the sup-norm of automorphic forms is still quite mysterious. For conjectures and results concerning the sup-norm problem we refer to \cite{BHMM16,Sa15} and the references within. Let us just say that there are many results towards improved upper bounds for the sup-norm in various settings. However, there are also some results giving lower bounds. On non-compact surfaces these lower bounds can have two sources. First they can come from the transition region of Whittaker functions. In this case the peak appears far from the so-called bulk of the manifold. Such lower bounds were considered for example in \cite{Sa15_2} and \cite{TB14}. Second, they can appear in the bulk of the manifold, which also appears in the compact setting. In this case we refer to \cite{MB16} for more details. The reason for mentioning these two sources for lower bounds is that we will encounter them in some sense in our arguments. Roughly speaking our proof will consist of two main parts. First we will obtain bounds by estimating Whittaker expansions. Second we use the amplification method to obtain good bounds in the bulk.

Now we dive straight into a more detailed set-up of the scene, so that we can give precise statements of our main theorems.

\textbf{Acknowledgements.} I would like to thank A. Saha and A. Booker helpful discussions and valuable feedback.

\subsection{Set-up and basic definitions}

Let $F$ be a number field of degree $n=r_1+2r_2$, where $r_1$ is the number of real embeddings and $2r_2$ is the number of complex embeddings. By $\Of$ we denote the ring of integers in $F$. Prime ideals in $\Of$ are typically denoted by $\p$. Each prime ideal gives rise to a non-archimedean place of $F$ which we also denote by $\p$. The corresponding (canonically normalized) valuation (used for field elements and ideals interchangeably) will be called $v_{\p}(\cdot)$ and gives rise to the absolute value $\abs{\cdot}_{\p} = q_{\p}^{-v_{\p}(\cdot)}$, where $q_{\p} = \No(\p)$. In a similar spirit we use $\nu$ for an archimedean place and at the same time for the corresponding embedding $\nu \colon F\to F_{\nu}$. We put $\abs{\cdot}_{\nu} = \abs{\cdot}^{[F_{\nu}:\R]}$. Here and in the following $\abs{\cdot}$ always denotes the standard absolute value on $F,\R \subset \C$. We equip the archimedean fields with the standard Lebesgue measure coming from $\R$ or $\C$, respectively.

If $F_{\p}$ is the local non-archimedean field associated to $\p$ then we write $\op_{\p}$ for its ring of integers and $\varpi_{\p}$ for its uniformizer. These fields are equipped with two measures. First the Haar measure $\mu_{\p}$ on $(F,+)$, which we normalize such that $\mu_{\p}(\op_{\p})=1$. Further we have the Haar measure $\mu^{\times}_{\p}$ on $(F^{\times},\cdot)$. This will be normalized to satisfy $\mu_{\p}^{\times}(\op_{\p}^{\times})=1$.

We define $F_{\infty} = \prod_{\nu} F_{\nu}$ and equip it with the modulus $\abs{\cdot}_{\infty} = \prod_{v}\abs{\cdot}_{\nu}$. Sometimes we use $\abs{\cdot}_{\R}$ (respectively $\abs{\cdot}_{\C}$) to denote the part of $\abs{\cdot}_{\infty}$ coming from the real (respectively complex) embeddings only. Let $\Afin$ denote the finite ad\'eles equipped with the absolute value $\abs{\cdot}_{\text{fin}}$ being the product of all the local absolute values. The usual ad\'ele ring is then given by $\A_F = F_{\infty}\times \Afin$ and equipped with $\abs{\cdot}_{\A}$ and $\mu$ in the usual manner. We also define the set of totally positive field elements $F^+ $ to contain all $x\in F$ such that $x_{\nu} >0$ for all real $\nu$. Furthermore put $F^0(\A_F)=\{ a\in \A_F\colon \abs{a}_{\A_F}=1 \}$ and $F_{\infty}^+= \R^+ \subset F_{\infty}$ diagonally.

Further let us choose ideal representatives $\theta_1, \dots, \theta_{h_F}\in \hat{\mathcal{O}}_{F}$, where $h_F$ denotes the class number of $F$. We write $d_F$ for the discriminant of $F$ and $\mathfrak{d}$ for the different ideal of $F$. Then by \cite[Theorem~2.9]{Ne13} we have $\No(\mathfrak{d}) = \abs{d_F}$.  For any ideal $\mathfrak{m}$ we use $[\mathfrak{m}]_{\n}= \frac{\mathfrak{m}}{(\mathfrak{m},\n^{\infty})}$ for the coprime-to-$\n$ part of $\mathfrak{m}$. 

If $\chi\colon F^{\times}\setminus \A_F^{\times} \to \C$ is a character we associated the corresponding $L$-function
\begin{equation}
	\Lambda(s,\chi) = \underbrace{\prod_{\nu} L_{\nu}(s,\chi_{\nu})}_{= \gamma_{\infty}(s,\chi)} \underbrace{\prod_{\p} L_{\p}(s,\chi_{\p}) }_{=L(s,\chi)}. \nonumber
\end{equation}
Here we use the classical (analytic) notation which separates the archimedean and non-archimedean parts. A complete description of the local factors can be found in \cite{La13}. If $\chi$ is the trivial character this leads to the Dedekind zeta function. In this case the local factors reduce to $\zeta_{\p}(s) = L_{\p}(1,s)=  (1-q_{\p}^{-s})^{-1}$ and we write $\zeta_{\n}(s) = \prod_{\p\mid\n} \zeta_{\p}(s)$. At the archimedean places we have $L_{\nu}(1,s) = \Gamma_{\R}(s) = \pi^{-\frac{s}{2}}\Gamma(\frac{s}{2})$ if $\nu$ is real and $L_{\nu}(1,s) =\Gamma_{\C}(s) = 2(2\pi)^{-s}\Gamma(s)$ otherwise.

Let $R$ be a commutative ring with 1. Typically this will be one of the objects introduced above. Then we set $G(R) = GL_2(R)$. We will also need the subgroups
\begin{eqnarray}
	Z(R) = \left\{  z(r)=\cent{r} \colon r\in R^{\times} \right\},& \quad A(R) = \left\{  a(r)=\am{r} \colon r\in R^{\times} \right\},  \nonumber\\
	N(R) = \left\{  n(x)=\nilp{x} \colon x\in R \right\}& \text{ and } \quad B(R) = Z(R)A(R)N(R). \nonumber
\end{eqnarray}
We use the following compact subgroups of $G(R)$ which depend on the underlying ring $R$. Define
\begin{eqnarray}
	G(F_{\nu}) \supset  K_{\nu} &=& \begin{cases} U_2(\C) & \text{ if $\nu$ is complex,} \\ O_2(\R) & \text{if $\nu$ is real,} \end{cases} \nonumber\\
	G(F_{\p}) \supset K_{\p} &=& GL_2(\op_{\p}), \nonumber \\
	K_{\infty} &=& \prod_{\nu} K_{\nu}. \nonumber 
\end{eqnarray}
At the non-archimedean places we additionally need the smaller groups
\begin{eqnarray}
	K_{\p}^0(n) &=& K_{\p} \cap \left[ \begin{matrix} \op_{\p}&\varpi_{\p}^n\op_{\p} \\ \op_{\p} &\op_{\p}\end{matrix} \right], K_{0,\p}(n) = K_{\p} \cap \left[ \begin{matrix} \op_{\p}&\op_{\p} \\ \varpi_{\p}^n\op_{\p} &\op_{\p}\end{matrix} \right] \text{ and }\nonumber\\
	K_{1,\p}(n) &=& K_{\p} \cap \left[ \begin{matrix} 1+\varpi_{\p}^n\op_{\p}&\op_{\p} \\ \varpi_{\p}^n\op_{\p} &\op_{\p}\end{matrix} \right], K_{2,\p}(n) = K_{\p} \cap \left[ \begin{matrix} \op_{\p}&\op_{\p} \\ \varpi_{\p}^n\op_{\p} &1+ \varpi_{\p}^n\op_{\p}\end{matrix} \right]. \nonumber
\end{eqnarray}
Globally we put
\begin{equation}
	K_1(\n) = K_{\infty} \prod_{\p} K_{1,\p}(v_{\p}(\n)) \text{ and } K=K_{\infty} \prod_{\p}K_{\p}. \nonumber
\end{equation}
Further let 
\begin{equation}
	\omega=\left(\begin{matrix} 0&1 \\ -1&0\end{matrix}\right) \nonumber
\end{equation}
be the long Weyl element.

Let us briefly describe the measures on the groups in use. Locally we will stick to the measure convention from \cite{Sa15, Sa15_2}. This means, we use the identifications $N(R) = (R,+)$, $A(R) = R^{\times}$, and $Z(R) = R^{\times}$ to transport the measures defined on the local fields to the corresponding groups. Further we take $\mu_{K_{\p}}$ to be the probability Haar measure on $K_{\p}$.
Globally we choose the product measure on $K$ and $A(\A_F)$ coming from the previously defined local measures. On the group $N(\A_F) = \A_F$ we put the measure
\begin{equation}
	\mu_{N(\A_F)} = \frac{2^{r_2}}{\sqrt{\abs{d_F}}} \prod_{\nu} \mu_{\nu} \prod_{\p} \mu_{\p}. \nonumber
\end{equation}
This corresponds to the normalization $\vol(N(F) \setminus N(\A_F)) = 1$, as can be seen from  strong approximation together with \cite[Chapter~I, Proposition~5.2]{Ne13}. Here it is important to be aware of the convention in \cite[p. 211]{Ne13} while identifying the Minkowski space with $F_{\infty}$. 
Finally we define 
\begin{equation}
	\int_{Z(\A_F)\setminus G(\A_F)} f(g) d\mu(g) = \int_K \int_{\A_F^{\times}} \int_{N(\A_F)} f(na(y) k) d\mu_{N(\A_F)}(n) \frac{d\mu^{\times}_{\A_F^{\times}}(y)}{\abs{y}} d\mu_K(k) \label{eq:integral_convention}
\end{equation}
as in \cite{GJ79}.

In this text we are interested in bounding the sup-norm of Maa\ss\  newforms over $F$ which are spherical at infinity. More precisely we will study functions 
\begin{equation}
	\phi\in L_0^2(G(F)\setminus G(\A_F),\omega) \subset L^2(G(F)\setminus G(\A_F),\omega) \nonumber
\end{equation}
which are right $K_1(\n)$-invariant, and eigenfunctions of the Casimir element $(C_{\nu})_{\nu}\in \mathcal{U}(\mathfrak{q}_{\infty})$ with eigenvalues $(\lambda_{\nu})_{\nu}$. These are automorphic forms in the sense of \cite[4.2]{BJ79}. Thus, it is standard procedure to associate an cuspidal automorphic representation\footnote{We use the definition of an automorphic representation given in \cite[4.6]{BJ79}. In particular irreducibility is included in the definition.} $\pi_\phi$ to $\phi$. As explained in \cite[4.6]{BJ79} each cuspidal automorphic representation with central character $\omega$ can be (uniquely) realized as a closed invariant subspace of $L_0^2(G(F)\setminus G(\A_F),\omega)$. In this way the problem of estimating the sup-norm of the Maa\ss\  newform $\phi$ is closely linked to the underlying cuspidal automorphic representation $\pi_{\phi}$. However, the sup-norm itself is only defined for smooth elements in $L_0^2(G(F)\setminus G(\A_F),\omega)$ and it does not make sense in different realizations of $\pi_{\phi}$. Therefore we will make the following convention.

\begin{conv}
Let $(\pi,V_{\pi})$ be a cuspidal automorphic representation with central character $\omega_{\pi}$. Then there is an intertwiner $\sigma\colon V_{\pi}\into L_0^2(G(F)\setminus G(\A_F),\omega)$. Then the \textbf{sup-norm} of a $K$-finite vector $v\in V_{\pi}$ is defined to be 
\begin{equation}
	\Vert v \Vert_{\infty} = \frac{\Vert \sigma(v)\Vert_{\infty}}{\Vert\sigma(v)\Vert_2}. \nonumber
\end{equation}
\end{conv}

Let us make some remarks concerning this convention.
\begin{itemize}
\item First note that this is indeed well defined. First we observe that by \textit{multiplicity one} for $GL_2$ the intertwiner $\sigma$ is unique up to scaling. However, the scaling does not matter since we $L^2$-normalize the image. Secondly $K$-finiteness ensures that the $L^{\infty}$-norm of $\sigma(v)$ is defined.
\item This convention may seem unnecessary at first. But it gives us the flexibility to realize $\pi$ in arbitrary models without changing the fixed cusp form whose sup-norm we want to bound.
\item The restriction to $K$-finite vectors shows that we should actually work with the $G(\A_F)$-module underlying $\pi$. 
\end{itemize}
 
Let us now describe the structure of the cuspidal automorphic representation $\pi$, keeping in mind that we are mainly interested in spherical Maa\ss\  newforms. We write $V_{\pi}$ for the representation space of $\pi$. First note that since $(\pi,V_{\pi})$ is an cuspidal automorphic representation it is in particular unitary and admissible. For convenience we assume throughout the text that the central character $\omega_{\pi}$ of $\pi$ satisfies $\omega_{\pi}\vert_{F_{\infty}^+}=1$. This can be achieved without loss of generality by twisting by an unramified character.

By the tensor product theorem \cite[Theorem~4]{Fl79} we may assume  that
\begin{equation}
	\pi = \bigotimes_{\nu} \pi_{\nu}\otimes \bigotimes_{\p} \pi_{\p}. \nonumber
\end{equation}
Where $(\pi_{\p},V_{\pi,\p})$ (respectively $(\pi_{\nu},V_{\pi,\nu})$ ) is an irreducible representation of $G(F_{\p})$ (reps $G(F_{\nu})$) with central character $\omega_{\pi,\p}$ (respectively $\omega_{\pi,\nu}$). Note that this decomposition also preserves the subspaces of $K$-finite vectors.  

Since we are only interested in automorphic forms which are spherical eigenfunctions of the Casimir operator we can restrict ourselves to a very particular situation at the archimedean places. Indeed we will always assume that $\pi_{\nu} = \chi_+\boxplus \chi_-$ with
\begin{eqnarray}
	\chi_j(y) &=&  \abs{y}_{\nu}^{ it_{\nu,j}}\sgn(y)^{m_{\nu}} \text{ if $\nu$ is real,} \nonumber \\
	\chi_j(re^{i\theta})  &=& r^{ i2t_{\nu,j}} e^{i m_{\nu} \theta} \text{ else.} \nonumber
\end{eqnarray}  
These are principal series representations and the representation space is denoted by 
\begin{equation}
	V_{\pi}=\mathcal{B}(\chi_+,\chi_-). \nonumber
\end{equation}
We define the invariants 
\begin{equation}
	t_{\nu} = (t_{\nu,1}-t_{\nu,2})/2 \text{ and } s_{\nu} = t_{\nu,1}+t_{\nu,2}.
\end{equation}
In particular we have $\omega_{\pi,\nu}\vert_{F_{\infty}^{+}}=\prod_{\nu} \abs{\cdot}_{\nu}^{s_{\nu}}$. Thus, the assumption $\omega_{\pi}\vert_{F_{\infty}^+}=1$ yields $\sum_{\nu} [F_{\nu}:\R]s_{\nu} = 0$.  Furthermore, $v\in V_{\pi,\nu}$ is an eigen vector of the Casimir operator with eigenvalue
\begin{equation}
	\lambda_{\nu} = \begin{cases} \frac{1}{4}+t_{\nu}^2 &\text{ if  $\nu$ is real,} \\ 1+4t_{\nu}^2 &\text{ else.}\end{cases} \nonumber
\end{equation}
This justifies calling $t_{\nu}$ the spectral parameter of $\pi$. For convenience we will exclude exceptional Maa\ss\  forms and will therefore assume $t_{\nu} \in \R$.

Note that a representation $\pi$ featuring these types of representations at the archimedean places are spherical. In other words, each representation $(\pi_{\nu},V_{\pi}$ contains a $K_{\nu}$-invariant vector $v_{\nu}^{\circ}$ which is unique up to scaling.
At the non-archimedean places we define $n_{\p}$ to be the log-conductor of $\pi_{\p}$. Then there exists a unique up to scaling vector $v_{\p}^{\circ}\in V_{\pi,\p}$ which is $K_{1,\p}(n_{\p})$-invariant.
Globally we define the conductor of $\pi$ to be the ideal $\n = \prod_{\p} \p^{n_{\p}}$. Thus, $V_{\pi}$ contains a unique (up to scaling) vector which is $K_1(\n)$-invariant. The vector
\begin{equation}
	v^{\circ} = \bigotimes_{\nu} v_{\nu}^{\circ}\otimes \bigotimes_{\p} v_{\p}^{\circ} \nonumber
\end{equation}
does the job and we will call it the (global) new vector. 

With this restrictions on $\pi$ in place we observe that
\begin{equation}
	 \phi_{\circ} = \frac{\sigma(v^{\circ})}{\Vert \sigma(v^{\circ})\Vert_2}
\end{equation}
is a Hecke-Maa\ss\  newform over $F$ with central character $\omega_{\pi}$ which is $K_1(\n)$-invariant and has Casimir eigenvalue $(\lambda_{\nu})_{\nu}$. Furthermore, by our convention $\Vert v^{\circ} \Vert_{\infty} = \Vert \phi_{\circ}\Vert_{\infty}$. This is exactly the setting in which we will study sup-norm problem. It is the natural generalization of classical Maa\ss\  wave forms on the upper half plane $\mathbb{H}$.

\subsection{Statement of results}

The theorems we state deal with newforms coming from cuspidal automorphic representations $(\pi,V_{\pi})$. Assume that at the infinite places $\pi$ features spherical principal series representations with spectral data $(t_{\nu})_{\nu}$. We associate the numbers $T=(T_{\nu})_{\nu}$ defined by
\begin{equation}
	T_{\nu} = \max\left(\tfrac{1}{2},\abs{t_{\nu}}\right). \nonumber
\end{equation}
To capture the arithmetic properties of $\pi$ we define the following ideals. First let $\mathfrak{m}$ be the conductor of the central character $\omega_{\pi}$. Further let $\n=\n_2\n_0^2$ be the conductor of $\pi$, where $\n_2$ denotes the square-free part of $\n$.

\begin{theorem} \label{th:main_th_1}
Let $(\pi,V_{\pi})$ be a cuspidal automorphic representation with conductor $\n$ and spectral parameter $(t_{\nu})_{\nu}$. And let $v^{\circ}$ be a new vector of $\pi$. Then
\begin{eqnarray}
	\frac{\sigma(v^{\circ})(g)}{\Vert \sigma(v^{\circ})\Vert_2} &\ll_{F,\epsilon}& (\abs{T}_{\infty} \No(\n))^{\epsilon} \bigg(\abs{T}_{\infty}^{\frac{5}{12}}\No(\n_2)^{\frac{1}{3}}\No(\n_0)^{\frac{1}{2}}\No(\mathfrak{m})^{\frac{1}{2}} \nonumber \\
	&&\qquad\qquad + \abs{T}_{\R}^{\frac{1}{4}}\abs{T}_{\C}^{\frac{1}{2}}\No(\n_2)^{\frac{1}{4}}\No(\n_0)^{\frac{1}{2}}\No(\mathfrak{m})^{\frac{1}{2}} \bigg). \nonumber
\end{eqnarray}
\end{theorem}   

\begin{rem}
Note that if we take $F=\Q$ this theorem reduces to \cite[Theorem~3.2]{Sa15}. On the other hand if $F$ is an arbitrary number field but $\n$ is square-free then we recover \cite[Theorem~1]{BHMM16}. 
\end{rem}

This theorem fails to meet our expectations for non-totally real fields $F$. Therefore, we will prove a second theorem generalizing \cite[Theorem~2]{BHMM16}. 

\begin{theorem} \label{th:main_th_2}
Let $F$ be number field with maximal totally real subfield $F^{\R}$ such that $[F:F^{\R}] = m\geq2$. For a cuspidal automorphic representation $(\pi, \sigma)$ with conductor $\n$ and spectral parameter $(t_{\nu})_{\nu}$ we have
\begin{equation}
	\frac{\sigma(v^{\circ})(g)}{\Vert \sigma(v^{\circ})\Vert_2} \ll_{F,\epsilon} (\abs{T}_{\infty} \No(\n))^{\epsilon} \abs{T}_{\infty}^{\frac{1}{2}-\frac{1}{8m-4}}\No(\n_2)^{\frac{1}{2}-\frac{1}{8m-4}} \No(\n_0)^{\frac{1}{2}}\No(\mathfrak{m})^{\frac{1}{2}} \nonumber
\end{equation}
where $v^{\circ}$ is a new vector.
\end{theorem}

At this point let us briefly add some speculations concerning the correct local bound. For compact locally symmetric spaces it seems quite clear what is correct local bound for the sup-norm of eigenfunctions. See for example \cite{letter}. The non compact situation seems to be much more complicated. While in a fixed compact set the same local bounds as for compact manifolds remain true, it has been shown in \cite{TB14} that globally these upper bounds are false. Thus, it still remains unclear what one should consider to be the correct local bound. We suggest that the right answer might come from the analysis of elliptic PDE's with boundary condition on manifolds with corners. This is due to the fact that one my compactify the underlying locally symmetric space (see \cite{BL05}). The resulting compact space will have a boundary. Then the fact that we are usually dealing with cusp forms suggests that one should look at solutions to elliptic operators with Dirichlet boundary conditions.

\subsection{Guide to the rest of the paper}

Let us now briefly give an overview over the rest of the paper. In Section~\ref{sec:reduction_step} we find a nice generating domain for $G(\A_F)$ which is tailor made for the transformation behaviour of $\phi_{\circ}$. 

We then move on towards the study of Whittaker functions associated to newforms. This will take up most of Section~\ref{sec:nounds_via_Whittaker} and culminate in the first upper bounds which are good towards the cusps.

The next step is to define an integral operator which will serve as an approximate spectral projector. This operator will then lead to what is usually called an amplified pre-trace formula. The geometric side of this pre-trace inequality can then be estimated as in \cite{BHMM16}. 

Finally, in Section~\ref{sec:endgame}, we will give complete proofs for the theorems stated above.

\section{The reduction step} \label{sec:reduction_step}

In this section we follow \cite[Section~3.2]{Sa15} to derive a generating domain for 
\begin{equation}
	Z(\A_F)G(F)\setminus G(\A_F) / K_1(\n). \nonumber
\end{equation}
We then continue from there and show that in order to solve the sup-norm problem for the automorphic forms under consideration we only have to bound our functions (and possibly their twists) on very special elements in $G(\A_F)$. The central results of this section are Corollary~\ref{cor:reduction_to_f_dom} below.

\subsection{Local preliminaries} \label{sec:lprem}

Several steps that are necessary to deal with powerful level rely on local methods. In this section we briefly recall the ingredients needed from \cite{Sa15}.

Let $\p$ be a finite place and let $(\pi_{\p}, V_{\pi,\p})$  denote an admissible, irreducible representation of $(F_{\p})$. Define $n_{\p} = a(\pi_{\p})$ to be the log-conductor of $\pi_{\p}$. Let $\omega_{\pi,\p}$ be the central character of $\pi_{\p}$ and let $m_{\p} = a(\omega_{\pi,\p})$ be its log-conductor. 

Most local computations rely on the decomposition
\begin{equation} \label{eq:decomp}
	G(F_{\p}) = \bigsqcup_{t\in \Z} \bigsqcup_{0\leq l\leq n_{\p}} \bigsqcup_{v\in \op_{\p}^{\times}/(1+\varpi_{\p}^{min(l,n_{\p}-l
)}\op_{\p}^{\times})} Z(F_{\p})N(F_{\p})\underbrace{a(\varpi_{\p}^t)\omega n(\varpi_{\p}^{-l}v)}_{=g_{t,l,v}}K_{1,\p}(n). 
\end{equation}
This is \cite[(3)]{Sa15} or originally \cite[Lemma~2.13]{Sa15_2}. This decomposition suggests to define the invariants $t_{\p}(g)$, $l_{\p}(g)$ and $n_{0,\p}(g)$ in the obvious way by writing
\begin{equation}
	g\in Z(F_{\p})N(F_{\p})g_{t(g),l(g),v}K_{1,\p}(n_{\p}) \nonumber
\end{equation}
with $v\in \op^{\times}/(1+\varpi_{\p}^{n_{0,\p}(g)}\op^{\times})$. We further define 
\begin{eqnarray}
	n_{1,\p}&=&\left\lceil \frac{n_{\p}}{2} \right\rceil, \nonumber\\ 
	n_{0,\p} &=& n_{\p}-n_{1,\p}, \nonumber\\
	n_{1,\p}(g) &=& \begin{cases} n_{0,\p} \text{ if } l(g)\leq n_{0,\p} ,\\   n_{1,\p} \text{ if } l(g)\geq n_{1,\p},\end{cases}\nonumber\\
	m_{1,\p}(g) &=& \max(0,n_{0,\p}(g)-n_{\p}+m_{\p}).\nonumber
\end{eqnarray}

Let us collect some simple results capturing the behavior of this invariants in crucial situations.

\begin{lemma} \label{lm:claim1}
Let $g\in K_{\p}a(\varpi_{\p}^{n_{1,\p}})$. If $n_{\p}$ is odd then 
\begin{equation}
	n_{1,\p}(g)=n_{0,\p} \iff g\in \omega K_{\p}^0(1)a(\varpi_{\p}^{n_{1,\p}}). \nonumber
\end{equation}
If $n_{\p}$ is even then
\begin{equation}
	n_{1,\p}(g)=n_{0,\p}. \nonumber
\end{equation}
\end{lemma}
\begin{proof}
The first part is a consequence of \cite[Lemma~2.2,(2)]{Sa15}. The second part holds since for even $n_{\p}$ one has $n_{0,\p}=n_{1,\p}$.
\end{proof}

\begin{lemma} \label{lm:claim2}
Let $n_{\p}$ be odd. Further take $k\in K_{0,\p}(1)$ and
\begin{equation}
	\epsilon_{\p} \in \left\{ 1, \left(\begin{matrix} 0&1\\ \varpi_{\p} &0\end{matrix}\right) \right\}. \nonumber
\end{equation} 
Then
\begin{equation}
	k\epsilon_{\p} \omega a(\varpi_{\p}^{n_{1,\p}}) = \omega k' a(\varpi_{\p}^{n_{1,\p}})\epsilon_{\p}' z \nonumber
\end{equation}
for $k'\in K_{\p}^0(1)$, $z\in Z(F_{\p})$ and
\begin{equation}
	\epsilon_{\p}' = \begin{cases} 1 &\text{ if } \epsilon_{\p}=1, \\ \left(\begin{matrix} 0&1 \\ \varpi_{\p}^{n_{\p}} &0 \end{matrix}\right) &\text{ else.} \end{cases} \nonumber
\end{equation}
\end{lemma}
This is \cite[Lemma~2.3]{Sa15}.
\begin{proof}
The case $\epsilon_{\p}=1$ is very simple. One writes
\begin{equation}
	k\epsilon_{\p}\omega a(\varpi_{\p}^{n_{1,\p}})= \omega \underbrace{(\omega^{-1} k \omega)}_{=k'} a(\varpi_{\p}^{n_{1,\p}}). \nonumber
\end{equation}
It is a straight forward calculation to check $k'\in K_{\p}^0(1)$. 
In the remaining case we write
\begin{equation}
	k\epsilon_{\p}\omega a(\varpi_{\p}^{n_{1,\p}}) =  \omega \underbrace{(\omega^{-1} k \omega)}_{=k'}a(\varpi_{\p}^{n_{1,\p}}) \epsilon_{\p}'\left( \begin{matrix} -\varpi_{\p}^{n_{0,\p}} &0 \\ 0& -\varpi_{\p}^{n_{0,\p}} \end{matrix} \right). \nonumber
\end{equation}
As before we have $k'\in K_{\p}^0(1)$. To verify the equality one only needs the observation that since $n_{\p}$ is odd we have $n_{0,\p} = n_{1,\p}-1$.
\end{proof}

\subsection{The global generating domain} \label{sec:gendom}

Our goal is to recreate the argument from \cite[Section~3.2]{Sa15} coupled with the results from \cite[Section~5]{BHMM16} to deal with arbitrary number fields. As one expects this general setting brings the class group and the unit group into the picture. We start with several definitions. For any ideal $\mathfrak{L}$ in $\mathcal{O}_{F}$ we define
\begin{eqnarray}
	\eta_{\mathfrak{L}} &=& \prod_{\p\mid \mathfrak{L}}\left( \begin{matrix} 0&1\\ \varpi_{\p}^{n_{\p}} & 0 \end{matrix} \right)\prod_{\p\nmid \mathfrak{L}} 1, \nonumber \\
	h_{\mathfrak{L}} &=& \prod_{\p\mid \mathfrak{L}} a(\varpi_{\p}^{n_{1,\p}})\prod_{\p\nmid \mathfrak{L}} 1, \nonumber \\
	K_{\mathfrak{L}} &=& \prod_{\p\mid \mathfrak{L}} K_{\p} \prod_{\p\nmid \mathfrak{L}} \{ 1 \} \subset GL_2(\Afin),  \nonumber\\
	J_{\mathfrak{L}} &=& K_{\mathfrak{L}}h_{\mathfrak{L}} \text{ and } \nonumber\\
	\mathcal{J}_{\mathfrak{L}} &=& \{ g\in J_{\mathfrak{L}} \colon n_{1,\p}(g_{\p}) = n_{0,\p} \forall \p\mid \mathfrak{L}\}.\nonumber
\end{eqnarray}

Before we proceed let us make the following little observation.

\begin{lemma} \label{lm:claim3}
For $g\in J_{\mathfrak{L}}$ one has 
\begin{equation}
	g\in \mathcal{J}_{\mathfrak{L}} \iff g_{\p} \in \omega K_{\p}^0(1)a(\varpi_{\p}^{n_{1,\p}}) \text{ for all } \p\vert \mathfrak{L} \text{ with } n_{\p}  \text{ odd.} \nonumber
\end{equation}
\end{lemma}
\begin{proof}
The proof proceeds by applying Lemma~\ref{lm:claim1} for each $\p\vert \mathfrak{L}$.
\end{proof}

\begin{cor} \label{cr:absorbtion}
For $g_{\p} \in \mathcal{J}_{\p}$ and $v\in \op_{\p}^{\times}$ we have $a(v)g \in \mathcal{J}_{\p}$.
\end{cor}
\begin{proof}
Obviously $a(v)g_{\p}\in J_{\p}$. One then concludes using Lemma~\ref{lm:claim3} and the fact
\begin{equation} 
a(v) \omega = \omega \left(\begin{matrix} 1&0 \\ 0& v\end{matrix}\right). \nonumber
\end{equation}  
\end{proof}

In terms of the local invariants we can write
\begin{equation}
	\n_0 = \prod_{\p} \p^{n_{0,\p}}, \n_1 = \prod_{\p} \p^{n_{1,\p}}, \n_2= \prod_{\p} \p^{n_{1,\p}-n_{0,\p}}. \nonumber
\end{equation}
Note that $\n_2$ is square free and that we have $\n=\n_0^2\n_2$.

Now we want to use the generating domain from \cite{BHMM16} for the square free ideal $\n_2$. Recall the group 
\begin{equation}
	K^* = Z(F_{\infty}) K_{\infty} \prod_{\p \nmid \n_2} Z(F_{\p})K_{\p} \prod_{\p\mid \n_2} \langle K_{0,\p}(1), \left(\begin{matrix} 0&1\\ \varpi_{\p} & 0\end{matrix}\right)\rangle. \nonumber
\end{equation}
defined in \cite[Section~2]{BHMM16}. Let $\mathcal{F}(\n_2)$ be the generating domain for $G(F)\setminus G(\A_F) /  K^*$ defined in \cite[p. 14]{BHMM16}. An element in $\mathcal{F}(\n_2)$ is of the form
\begin{equation}
	\underbrace{\left(\begin{matrix} y&x\\0&1\end{matrix}\right)}_{\in B(F_{\infty})}\underbrace{\left( \begin{matrix}\theta_i &0 \\ 0 &1\end{matrix}\right)}_{=a(\theta_i)}, \nonumber
\end{equation}
where $\abs{y}_{\infty}$ is maximal and $\theta_i\in \hat{\mathcal{O}}^{\times}_{F}$ is an ideal representative. We will call such matrices special.
Define
\begin{equation}
	\mathcal{F}_{\n_2} = \left\{ \left( \begin{matrix} y & x\\ 0&1 \end{matrix}\right) \colon \exists i\in \{1,\cdots,h \} \text{ such that } \left( \begin{matrix} y & x\\ 0&1 \end{matrix}\right)a(\theta_i)\in \mathcal{F}(\n_2)\right\}. \nonumber
\end{equation}

We can write down a generating domain in the spirit of \cite[Proposition~3.6]{Sa15}.

\begin{prop} \label{prop:generating_domain_expil}
For $g\in G(\A_F)$ we find $\mathfrak{L}\vert \n_2$ and $1\leq i\leq h_F$ such that
\begin{equation}
	g\in Z(\A)G(F) \left( a(\theta_i) \mathcal{J}_{\n} \times \mathcal{F}_{\n_2} \right)\eta_{\mathfrak{L}}K_1(\n). \nonumber
\end{equation}
\end{prop}
The proof follows exactly the steps in \cite{Sa15} exploiting that the fundamental domain $\mathcal{F}(\n_2)$ from \cite{BHMM16} is already given ad\'elically.
\begin{proof}
Let $\omega_{\n}$ be the diagonal embedding of $\omega$ in $K_{\n}$. Then the determinant map
\begin{equation}
	\omega_{\n}h_{\n} K_1(\n)_{fin}h_{\n}^{-1}\omega_{\n}^{-1} \to \prod_{\p}\op_{\p}^{\times} \nonumber
\end{equation}
is surjective. Thus we can apply strong approximation to the element $gh_{\n}^{-1}\omega_{\n}^{-1}$ and find $g_{\infty}t_i\in G(\A_F)$ such that
\begin{equation}
	g\in G(F)g_{\infty}t_i\omega_{\n}h_{\n}K_1(\n). \nonumber
\end{equation}
Using the properties of $\mathcal{F}(\n_2)$ we write $g_{\infty}t_i = \gamma f z k^*$ with $\gamma\in G(F)$, $zk^*\in K^*$ and $f\in \mathcal{F}(\n_2)$.
By construction of $K^*$ we can assume 
\begin{eqnarray}
	k^*_{\p} &=& \begin{cases} k^*_{\p} \in K_{\p} &\text{ if } \p\nmid \n_2, \\ k'_{\p}\epsilon_{\p} \in K_{0,\p}\epsilon_{\p} &\text{ if } \p\mid\n_2,  \end{cases} \nonumber \\
	k^*_v &\in&  K_v. \nonumber
\end{eqnarray}
for $\epsilon_{\p} \in \left\{ 1, \left( \begin{matrix} 0&1\\ \varpi_{\p} &0 \end{matrix}\right)\right\}$. We define 
\begin{equation}
	\mathfrak{L} = \prod_{\p \text{ s.t. } \epsilon_{\p}\neq 1} \p. \nonumber
\end{equation}
We can write
\begin{equation}
	g\in Z(\A)G(F) f \underbrace{\prod_{\p\nmid \n} k_{\p}^*}_{\in K_1(\n)}\prod_{\p\mid \n,\p\nmid\n_2} k_{\p}^*\omega a(\varpi_{\p}^{n_{1,\p}})\prod_{\p\mid \n_2,\p\nmid\mathfrak{L}} k_{\p}'\omega a(\varpi_{\p}^{n_{1,\p}})\prod_{\p\mid \mathfrak{L}} k_{\p}'\epsilon_{\p}\omega a(\varpi_{\p}^{n_{1,\p}})K_{1}(\n). \nonumber
\end{equation}
Let us treat each product appearing above separately. First we include the product over $\p\nmid\n$ into $K_1(\n)$. 
Next we notice that if $\p\mid \n$ but $\p \nmid \n_2$ then $n_{\p}$ must be even. Since obviously $k^*_{\p}\omega\in K_{\p}$ we apply Lemma~\ref{lm:claim1} to absorb the second product into $\mathcal{J}_{\n}$.
In the two remaining cases, namely $\p\mid\n_2$, $n_{\p}$ must be odd.  First for $\p\nmid\mathfrak{L}$ we apply Lemma~\ref{lm:claim2} to obtain
\begin{equation}
	k'_{\p}\omega a(\varpi_{\p}^{n_{1,\p}}) = \omega \underbrace{\hat{k}_{\p}}_{\in K_{\p}^o(1)}a(\varpi_{\p}^{n_{1,\p}}). \nonumber
\end{equation}
Now it follows from Lemma~\ref{lm:claim3} that also the third product is contained in $\mathcal{J}_{\n}$.
Finally we use Lemma~\ref{lm:claim2} and  Lemma~\ref{lm:claim3} again to get
\begin{equation}
	k'_{\p}\epsilon_{\p}\omega  a(\varpi_{\p}^{n_{1,\p}}) = \underbrace{\omega\overbrace{\hat{k}_{\p}}^{K_{\p}^0(1)} a(\varpi_{\p}^{n_{1,\p}})}_{\in\mathcal{J}_{\n}}\epsilon_{\p}'\hat(z). \nonumber
\end{equation}
Thus,
\begin{equation}
	g\in Z(\A)G(F) f  \mathcal{J}_{\n}\underbrace{\prod_{\p\vert\mathfrak{L}}\epsilon_{p}'}_{=\eta_{\mathfrak{L}}} K_1(\n). \nonumber
\end{equation}
One concludes the proof by writing $f=pa(\theta_i)$ for a special matrix $p\in \mathcal{F}_{\n_2}$ and some $i\in\{ 1\cdots h_F\}$.
\end{proof}

\subsection{The action of $\eta_{\mathfrak{L}}$}

The next step is to understand how the matrix $\eta_{\mathfrak{L}}$ acts on the automorphic functions under consideration. 

Let us define the character $\omega_{\pi}^{\mathfrak{L}} = \omega^{\mathfrak{L}}_{\pi,\infty}\prod_{\p}\omega_{\pi,\p}^{\mathfrak{L}}$ by
\begin{equation}
	\omega_{\pi,\p}^{\mathfrak{L}}\vert_{\op_{\p}^{\times}} = \begin{cases} 1 &\text{ if } \p\mid\mathfrak{L},\\ \omega_{\pi,\p}\vert_{\op_{\p}^{\times}} & \text{ if } \p \nmid \mathfrak{L}. \end{cases} \nonumber
\end{equation}
We also impose that $\omega^{\mathfrak{L}}_{\pi,\infty}\vert_{F_{\infty}} = 1$ and $\omega_{\pi,\p}(\varpi_{\p})=1$ for all $\p$. Strong approximation for $\A_F^{\times}$ shows that this defines an unique character which is $F^{\times}$ invariant.

Let us make some observations. Locally one has
\begin{equation} \label{eq:Lchar_cas}
	\omega_{\pi,\p}^{-1}\omega_{\pi,\p}^{\mathfrak{L}}\vert_{\op_{\p}^{\times}} = \begin{cases} \omega_{\pi,\p}^{-1}\vert_{\op_{\p}^{\times}} &\text{ if } \p\mid\mathfrak{L},\\ 1 & \text{ if } \p \nmid \mathfrak{L}. \end{cases}
\end{equation}

Now let $(\pi,V_{\pi})$ be a cuspidal automorphic representation. Then we define the twisted representation $(\pi^{\mathfrak{L}},V_{\pi})$ by
\begin{equation}
	\pi^{\mathfrak{L}}(g) = \omega_{\pi}^{-1}\omega_{\pi}^{\mathfrak{L}}(\det(g))\pi(g). \nonumber
\end{equation}
This representation is sometimes denoted by $\pi^{\mathfrak{L}} =(\omega_{\pi}^{-1}\omega_{\pi}^{\mathfrak{L}})\pi$. The central character of $\pi^{\mathfrak{L}}$ is $\omega_{\pi}^{-1}(\omega_{\pi}^{\mathfrak{L}})^2$ and looks locally like
\begin{equation} \label{eq:cent_cas}
	\omega_{\pi,\p}^{-1}(\omega_{\pi,\p}^{\mathfrak{L}})^2\vert_{\op_{\p}^{\times}} = \begin{cases} \omega_{\pi,\p}^{-1}\vert_{\op_{\p}^{\times}} &\text{ if } \p\mid\mathfrak{L},\\ \omega_{\pi,\p}\vert_{\op_{\p}^{\times}} & \text{ if } \p \nmid \mathfrak{L}. \end{cases}
\end{equation}
In particular the log-conductor of the new central character coincides with the log-conductor of $\omega_{\pi}$, namely
\begin{equation}
	\mathfrak{m}=\prod_{\p} \p^{m_{\p}}. \nonumber
\end{equation}
Further we note that this twist does not change the spectral data at $\infty$. Concerning the conductor of $\pi^{\mathfrak{L}}$ we have the following statement.

\begin{lemma} \label{lemma:actio_of_eta_on_newvec}
The log-conductor of $\pi^{\mathfrak{L}}$ is $\n$  and
\begin{equation} \label{eq:newvec}
	v^{\circ}_{\mathfrak{L}} =  \pi(\eta_{\mathfrak{L}})v^{\circ}
\end{equation}
is a new vector in $\pi^{\mathfrak{L}}$. 
\end{lemma}
This corresponds to \cite[Lemma~3.4]{Sa15}.
\begin{proof}
Note that for $\p\nmid \mathfrak{L}$ one simply has $\pi^{\mathfrak{L}}_{\p}=\pi_{\p}$. However at the places $\p\mid\mathfrak{L}$ the representation $\pi^{\mathfrak{L}}_{\p}$ is equivalent to $\tilde{\pi}_{\p}$ up to some unramified twist. Here $\tilde{\pi}_{\p}$ denotes the contragient representation of $\pi_{\p}$. Since $a(\pi_{\p}) = a(\tilde{\pi}_{\p})$ it suffices to show that  the vector given in \eqref{eq:newvec} has the correct transformation behavior under $K_1(\n)$.

We proceed place by place. For $\p\nmid\mathfrak{L}$ and $\nu$ there is nothing to do. 
For $\p\mid \mathfrak{L}$ we calculate
\begin{equation}
	\underbrace{\left(\begin{matrix} a&b\\c&d\end{matrix}\right)}_{=k_{\p}\in K_{1,\p}(n_{\p})}\left( \begin{matrix} 0&1\\ \varpi_{\p}^{n_{\p}} & 0 \end{matrix}\right) = \left( \begin{matrix} 0&1\\ \varpi_{\p}^{n_{\p}} & 0 \end{matrix}\right) \underbrace{\left(\begin{matrix} d& c\varpi_{\p}^{-n_{\p}}\\ \varpi_{\p}^{n_{\p}}b & a \end{matrix}\right)}_{=k_{\p}'\in K_{0,\p}(n_{\p})} \nonumber
\end{equation}
It is easy to verify that $k_{\p}'z(\det(k_{\p}))^{-1}\in K_{1,\p}(n_{\p})$. Therefore, using \eqref{eq:Lchar_cas} and \eqref{eq:cent_cas} we have
\begin{eqnarray}
	\pi_{\p}^{\mathfrak{L}}(k_{\p}) v^{\circ}_{\mathfrak{L},\p} &=& \omega_{\pi,\p}^{-1}(\det(k_{\p})) \pi_{\p}(k_{\p}[\eta_{\mathfrak{L}}]_{\p})v_{\p}^{\circ} \nonumber \\
	&=& \omega_{\pi,\p}^{-1}\big(\det(k_{\p})\big) \underbrace{\pi_{\p}\big(z(\det(k_{\p}))\big)}_{= \omega_{\p}\big(\det(k_{\p})\big)}  \pi_{\p}\big([\eta_{\mathfrak{L}}]_{\p}\big) \big[\underbrace{\pi_{\p}\big(\underbrace{z(\det(k_{\p})^{-1})k_{\p}'}_{\in K_{1,\p}(n_{\p})}\big)v_{\p}^{\circ}}_{=v_{\p}^{\circ}}\big]  \nonumber\\
	&=& \pi_{\p}\big([\eta_{\mathfrak{L}}]_{\p}\big)v_{\p}^{\circ} = v_{\mathfrak{L},\p}^{\circ}. \nonumber
\end{eqnarray} 
\end{proof}

Observe that $(\pi^{\mathfrak{L}},V_{\pi})$ is also a cuspidal automorphic representation. Furthermore, an intertwiner, $\sigma^{\mathfrak{L}}$, to $L_0^2(G(F)\setminus G(\A_F),\omega_{\pi}^{-1}(\omega_{\pi}^{\mathfrak{L}})^2)$ is given by 
\begin{equation}
	[\sigma^{\mathfrak{L}}(v)](g) = \omega^{-1}_{\pi}\omega_{\pi}^{\mathfrak{L}}(\det(g))[\sigma(v)](g). \nonumber
\end{equation}
This leads us to the definition of the twisted newform $\phi_{\circ}^{\mathfrak{L}}= \sigma^{\mathfrak{L}}(v_{\mathfrak{L}}^{\circ})$. One immediately observes that
\begin{equation}
	\phi_{\circ}(g\eta_{\mathfrak{L}}) = \omega_{\pi}(\omega_{\pi}^{\mathfrak{L}})^{-1}(\det(g)) \phi_{\circ}^{\mathfrak{L}}(g). \nonumber
\end{equation}

Giving us exactly the ingredient we needed to understand the action of $\eta_{\mathfrak{L}}$ on $\phi_{\circ}$. We derive the following corollary.

\begin{cor} \label{cor:reduction_to_f_dom}
If $\phi_{\circ}$ is the newform associated to a cuspidal automorphic representation $(\pi,V_{\nu})$ then
\begin{equation} \label{eq:red_to_psi}
	\sup_{g\in GL_2(\A)} \abs{\phi_{\circ}(g)} \leq \sup_{\mathfrak{L}\vert \n_2}\sup_{1\leq i \leq h_F}\sup_{g\in \mathcal{J}_{\n}\times \mathcal{F}_{\n_2}} \abs{\phi_{\circ}^{\mathfrak{L}}\left( a(\theta_i)g\right)}.
\end{equation}
\end{cor}

Therefore, we reduced the sup-norm problem for the newform $\phi_{\circ}$ to bounding the newforms $\phi_{\circ}^{\mathfrak{L}}$ on very special matrices. In the following we will fix an arbitrary $\mathfrak{L}\mid \n_2$ and write $\phi=\phi_{\circ}^{\mathfrak{L}}$. We can therefore focus on bounding $\phi$ at $a(\theta_i)(\mathcal{J}_{\n}\times \mathcal{F}_{\n_2})$.

\section{Bounds via Whittaker expansions} \label{sec:nounds_via_Whittaker}

In this section we consider the Whittaker expansion of cusp forms. This will lead to first upper bounds for the newform $\phi_{\circ}$ which is sufficient close to the cusps. The main result is Proposition~\ref{pr:whit_exp_est_bad} below.

Throughout this section let $(\pi, V_{\pi})$ be a cuspidal automorphic representation with new vector $v^{\circ}\in V_{\pi}$ and associated newform $\phi_{\circ}=\sigma(v^{\circ})$. Without loss of generality we assume that $\phi_{\circ}$ is $L^2$-normalized. Further we fix $g\in\mathcal{J}_{\n}$ and $n(x)a(y)\in \mathcal{F}_{\n_2}$.

\subsection{The Whittaker expansion of cusp forms}

Let $\psi=\prod_v \psi_v \prod_{\p}\psi_{\p}$ be the standard additive character of $\A_F$ as defined in \cite{Sa15_2}. Recall
\begin{equation}
	\psi_{\nu}(x) = \begin{cases} e(x) &\text{ if $\nu$ is real},\\ e(x+\overline{x}) &\text{ if $\nu$ is complex.} \end{cases} \nonumber
\end{equation}
Further note that the conductor of $\psi$ is $\mathfrak{d}^{-1}$. 

Having fixed the additive character we define the corresponding global Whittaker function
\begin{equation}
	W_{\phi_{\circ}}(g) = \frac{2^{r_2}}{\sqrt{d_F}}\int_{F\setminus \A_F} \phi_{\circ}(n(x)g)\psi(-x)d\mu_{\A_F}(x). \nonumber
\end{equation} 
We want to factor this global function into a product of local functions each of which matches the ones studied in \cite{Sa15}. Therefore, we define the shifted local characters $\psi_{\p}' = \psi_{\p}(\varpi_{\p}^{-v_{\p}(\mathfrak{d})}\cdot)$. This local additive character has conductor $\op_{\p}$. Further if $\omega_{\pi,\p}(\varpi_{\p}) = \abs{\varpi_{\p}}_{\p}^{ia_{\p}}$ we define $\pi_{\p}' = \abs{\cdot}_{\p}^{i\frac{a_{\p}}{2}}\pi_{\p}$. The purpose of this twist is that the central character $\omega_{\pi_{\p}}'$ of $\pi_{\p}'$ is trivial on the uniformizer. Now let $W_{\p}$ be the Whittaker new vector associated to the representation $\pi_{\p}'$ with respect to the character $\psi_{\p}'$ normalized by $W_{\p}(1)=1$. These are exactly the Whittaker functions studied in \cite{Sa15, Sa15_2}. At infinity we take local Whittaker function $W_{\nu}$ to be the Whittaker vector associated to $v_{\nu}^{\circ}$ normalized by $\SP{W_{\nu}}{W_{\nu}}=1$. This matches the situation in \cite{BHMM16}. Having defined these local functions we achieve the factorization
\begin{equation} 
	W_{\phi_{\circ}}(g) = c_{\phi_{\circ}}\underbrace{\prod_{\nu} W_{\nu}(g_{\nu})}_{=W_{\infty}(g_{\infty})} \prod_{\p} \abs{\det(g_{\p})}^{-i\frac{a_{\p}}{2}}_{\p} W_{\p}(a(\varpi_{\p}^{v_{\p}(\mathfrak{d})})g_{\p}). \nonumber
\end{equation}
The translation in the finite part comes from the shift in the local additive characters as explained in \cite[Remark~2.11]{Sa15_2}. The constant $c_{\psi}$ comes from our re-normalization of the local functions.

For $1\leq i \leq h_F$ and $g\in \mathcal{J}_{\n}$ we have the well known Whittaker expansion
\begin{equation}
	\phi_{\circ}(a(\theta_i)gn(x)a(y))= c_{\psi}\sum_{q \in F^{\times}} \prod_{\p}\abs{q\theta_i\det(g_{\p})}_{\p}^{-i\frac{a_{\p}}{2}} W_{\p}(a(\varpi_{\p}^{v_{\p}(\mathfrak{d})}\theta_i q)g_{\p})W_{\infty}(a(q)n(x)a(y)). \nonumber
\end{equation}

For convenience we split the local terms in the archimedean part $W_{\infty}$, the unramified part
\begin{equation}
	\lambda_{ur}(q) = \prod_{\p\nmid \n}  W_{\p}(a(\varpi_{\p}^{v_{\p}(\mathfrak{d})}\theta_i q)), \nonumber
\end{equation}
and the ramified part
\begin{equation}
	\lambda_{\n}(q) = \prod_{\p\mid \n}  W_{\p}(a(\varpi_{\p}^{v_{\p}(\mathfrak{d})}\theta_i q)g_{\p}). \nonumber
\end{equation}

We also collect all the unramified twists together and write $\eta(q) = \prod_{\p}\abs{q\theta_i\det(g_{\p})}_{\p}^{-i\frac{a_{\p}}{2}}$. Since $\abs{\eta} = 1$ this factor does not influence any of the upcoming estimates. 

Let us continue by gathering some properties of $\lambda_{\n}$ and $\lambda_{ur}$. First we recall the following standard result.

\begin{lemma} \label{lm:sph_wh}
If $\p\nmid \n$ then there are unramified characters $\chi_{1,\p}$ and $\chi_{2,\p}$ such that $\pi_{\p}' = \chi_{1,\p} \boxplus \chi_{2,\p}$. In this case we have
\begin{equation}
	W_{\p}(a(\varpi_{\p}^{v_{\p}(\mathfrak{d})}\theta_i q)) = \begin{cases} &0 \qquad\text{ if } v_{\p}(\theta_iq)+v_{\p}(\mathfrak{d}) < 0, \\  &q_{\p}^{-( v_{\p}(\theta_iq)+v_{\p}(\mathfrak{d}))/2} \frac{\chi_{1,\p}(\varpi_{\p})^{v_{\p}(\theta_iq)+v_{\p}(\mathfrak{d})+1}-\chi_{2,\p}(\varpi_{\p})^{v_{\p}(\theta_iq)+v_{\p}(\mathfrak{d})+1}}{\chi_{1,\p}(\varpi_{\p})-\chi_{2,\p}(\varpi_{\p})} \\ &\qquad\text{ if }  v_{\p}(\theta_iq)+v_{\p}(\mathfrak{d}) \geq 0. \end{cases} \nonumber
\end{equation}
\end{lemma}
\begin{proof}
This follows from \cite[Theorem~4.6.4]{Bu96} and \cite[Theorem~4.6.5]{Bu96}.
\end{proof}

We can extract the following fact about the support of unramified coefficients.

\begin{cor} \label{cor:supp_ur}
If $\lambda_{ur}(q) \neq 0$ then $v_{\p}(q)\geq -v_{\p}(\mathfrak{d})-v_{\p}(\theta_i)$ for all $\p\nmid \n$.
\end{cor}

In order to describe the unramified coefficients in terms of more ore less well known terms we quickly introduce the Hecke operators. For $\p\nmid \n$ and $k\in \N$ define
\begin{equation}
	X_{\p,k}=\{ m\in\text{Mat}_2(\op_{\p}) \colon v_{\p}(\det(m))=k\}. \nonumber
\end{equation}
The local new vector $v_{\p}^{\circ}$ is an eigenvector of the operator $\pi_{\p}(\mathbbm{1}_{X_{\p,k}})$ and we denote its eigenvalue by $\lambda(\p^k)$. For any ideal $\mathfrak{a}$ co-prime to $\n$ we define the global Hecke operator by $T(\mathfrak{a}) = \prod_{\p\vert\mathfrak{a}} \pi_{\p}(\mathbbm{1}_{X_{\p,v_{\p}(\mathfrak{a})}})$. It is clear that the global new vector $v^{\circ}$ and therefore also the newform $\phi_{\circ}$, is an eigenvector of this operator with eigenvalue $\lambda(\mathfrak{a})= \prod_{\p\mid\mathfrak{a}}\lambda(\p^{v_{\p}(\mathfrak{a})})$. We can now make a connection between $\lambda_{ur}$ and the Hecke eigenvalues $\lambda(\cdot)$. At this point let us remark that we follow the normalization of \cite[Section~4.6]{Bu96} which differs from the one used in \cite{Sa15} and \cite{BHMM16}.

\begin{lemma} \label{lm:unram_rel}
We have
\begin{equation}
	\lambda_{\text{ur}}(q) = \frac{\lambda\left(  \frac{(q)\theta_i\mathfrak{d}}{[(q)\theta_i\mathfrak{d}]_{\n}} \right)}{\No\left(\frac{(q)\theta_i\mathfrak{d}}{[(q)\theta_i\mathfrak{d}]_{\n}}\right)}. \nonumber
\end{equation}
\end{lemma}
\begin{proof}
The proof proceeds locally by showing
\begin{equation}
	\lambda(\p^k) = q_{\p}^k W_{\p}(a(\varpi^{k})) \text{ for } \p\nmid \n. \nonumber
\end{equation}
This can be done by induction using \cite[Proposition~4.6.4, Proposition~4.6.6]{Bu96} and Lemma~\ref{lm:sph_wh}.
\end{proof}

Next let us inspect the support of $\lambda_{\n}$.
\begin{lemma} \label{lm:supp_ram}
If $\lambda_{\n}(q)\neq 0$ then $v_{\p}(q) \geq  -v_{\p}(\theta_i)-v_{\p}(\mathfrak{d})-n_{0,\p}-m_{1,\p}(g_p)$ for all $\p\mid \n$.
\end{lemma}
This is close to \cite[Lemma~3.11]{Sa15}.
\begin{proof}
Since $g\in \mathcal{J}_{\n}$ we have $g_p\in K_{\p}a(\varpi_{\p}^{n_{1,\p}})$ and $n_{1,\p}(g_{\p})=n_{0,\p}$. But $W_{\p}(a(\varpi_{\p}^{v_{\p}(\mathfrak{d})}\theta_i q)g_{\p})\neq 0$ so that \cite[Proposition~2.11,(1)]{Sa15} implies\footnote{Note that in the notation of \cite{Sa15} we have $q(g_{\p}) = n_{0,\p}+m_{1,\p}(g_{\p})$.}
\begin{equation}
	v_{\p}(\theta_i q)+v_{\p}(\mathfrak{d}) \geq -n_{1,\p}(g_p)-m_{1,\p}(g_p). \nonumber
\end{equation}
Note that we used Corollary~\ref{cr:absorbtion} to include $a(v')$ into $g_{\p}$ for $v'\in\op_{\p}^{\times}$ where $\theta_i q = v' \varpi_{\p}^{v_{\p}(\theta_i q)}$. 
\end{proof}

Later on it will make sense to view $\lambda_{\n}$ as a locally constant function on the ad\'eles in an obvious way. It will then be crucial to determine sets on which this function is constant. 

\begin{lemma} \label{lm:constans}
Let $\p \vert \n$ and $u_1,u_2 \in \op_{\p}^{\times}$ such that $u_1-u_2\in \varpi_{\p}^{n_{0,\p}(g_{\p})}\op_{\p}$. Then
\begin{equation}
	\abs{W_{\pi_{\p}}(a(\varpi_{\p}^k u_1)g_{\p})} = \abs{W_{\pi_{\p}}(a(\varpi_{\p}^k u_2)g_{\p})}. \nonumber
\end{equation}
\end{lemma}
This is essentially \cite[Lemma~3.12]{Sa15}.
\begin{proof}
The proof of this little lemma goes back to the decomposition \eqref{eq:decomp} and the fact that $\abs{W_{\pi_{\p}}}$ is well defined by its values on $g_{t,l,v}$.

First let us write 
\begin{equation}
	g_{\p}= zng_{t,l,v} k. \nonumber
\end{equation}
Then one observes that
\begin{equation}
	a(\varpi_{\p}^k u_1)g_{\p} = zn'g_{t+k,l,vu_1^{-1}}k'. \nonumber
\end{equation}
By doing the same for $u_2$ we observe, that the claimed equality follows when
\begin{equation}
	[vu^{-1}] = [vu_2^{-1}] \in \op_{\p}^{\times} / (1+\varpi_{\p}^{n_0(g_{\p})}\op_{\p}). \nonumber
\end{equation}
The last condition leads to $u_1-u_2\in \varpi_{\p}^{n_0(g_{\p})}\op_{\p}$.
\end{proof}

Combining the support properties from Lemma~\ref{lm:supp_ram} and Corollary~\ref{cor:supp_ur} we derive
\begin{equation}
	\abs{\phi_{\circ}(a(\theta_i)gn(x)a(y))} \leq \abs{c_{\phi_{\circ}}}\sum_{q \in \imath^{-1}} \abs{\lambda_{ur}(q)\lambda_{\n}(q)W_{\infty}(a(qy))}. \label{eq:abs_of_Whitt_for_cusp}
\end{equation}
Where 
\begin{equation}
	\imath = \n_0\mathfrak{m}_1(g)\mathfrak{d} \prod_{\p} \p^{v_{\p}(\theta_i)}, \quad \mathfrak{m}_1(g) = \prod_{\p} \p^{m_{1,\p}(g_{\p})}. \label{eq:def_of_imath}
\end{equation}

It is easy to deal with the constant $c_{\phi_{\circ}}$.
\begin{lemma} \label{lm:the_scaling_const}
We have
\begin{equation}
	c_{\phi_{\circ}} \ll_{F,\epsilon} \big( \No(\n)\abs{T}_{\infty} \big)^{\epsilon}. \nonumber
\end{equation}
\end{lemma}
\begin{proof}
As in \cite{Sa15_2} we observe
\begin{equation}
	c_{\phi_{\circ}}^2 \ll_F L^{-1}(1,\pi,Ad)^{-1}\prod_{\nu}\SP{W_{\nu}}{W_{\nu}}^{-1} = L(1,\pi,Ad)^{-1}. \nonumber
\end{equation}
It is a well known fact that $L(1,\pi,Ad)\gg (\No(\n)\abs{T}_{\infty})^{\epsilon}$. Thus,
\begin{equation}
	c_{\phi_{\circ}} \ll (\No(\n)\abs{T}_{\infty})^{\epsilon}. \nonumber
\end{equation}
\end{proof}

Before continuing we fix a parameter $R=(R_{\nu})_{\nu}$ and define the box
\begin{equation}
	B(R) = \prod_{\nu} \{ \xi_{\nu}\in F_{\nu} \colon \abs{\xi_{\nu}} \leq R_{\nu} \}. \nonumber
\end{equation} 
This box will be used to truncate the Whittaker expansion. We will mostly use $R_{\nu} \asymp \frac{T_{\nu}}{y_{\nu}}$ except in Section~\ref{sec:est_S2} below, where we allow arbitrary $R$.

Applying the H\"older inequality together with $1=\abs{q}_{\A_F} = \abs{q}_{fin} \abs{q}_{\infty}$ yields
\begin{eqnarray}
	\abs{\phi_{\circ}(a(\theta_i)gn(x)a(y))} &\leq& \abs{c_{\phi_{\circ}}} \underbrace{\left( \sum_{q \in \imath^{-1}\cap B(R)} \abs{q}_{\infty}^{-2}\abs{W_{\infty}(a(qy))}^4 \right)^{\frac{1}{4}}}_{=S_1} \label{eq:whitt_after_ineq} \\
	&&\qquad \cdot \underbrace{\left( \sum_{q \in \imath^{-1}\cap B(R)} \No(q)^{\frac{2}{3}}\abs{\lambda_{ur}(q)\lambda_{\n}(q)}^{\frac{4}{3}} \right)^{\frac{3}{4}}}_{=S_2(R)} + \abs{c_{\phi_{\circ}}}\mathcal{E}. \nonumber
\end{eqnarray}
with 
\begin{eqnarray}
		S_1 &=& \left(\sum_{q \in \imath^{-1}\cap B(R)} \abs{q}_{\infty}^{-2}\abs{W_{\infty}(a(qy))}^4\right)^{\frac{1}{4}}, \nonumber \\
		S_2(R) &=& \left(\sum_{q \in \imath^{-1}\cap B(R)} \No(q)^{\frac{2}{3}}\abs{\lambda_{ur}(q)\lambda_{\n}(q)}^{\frac{4}{3}}\right)^{\frac{3}{4}}, \text{ and }\nonumber \\
	\mathcal{E} &=& \sum_{q\in \imath^{-1}, q\not\in B(R)} \abs{\lambda_{ur}(q)\lambda_{\n}(q)W_{\infty}(q y)}. \nonumber
\end{eqnarray}

We will estimate each one of these 3 quantities in the upcoming subsections.

\subsection{Counting field elements in boxes} \label{sec:counting_in_boxes}

This subsection is concerned with estimating the number of field elements in different ad\'elic boxes. These estimates will be needed in order to estimate $S_1$, $S_2(R)$, and $\mathcal{E}$. 

We start by considering some archimedean boxes. The following argument is almost completely taken from \cite{BHMM16}. Take  $R_{\nu} = \frac{T_{\nu}+T_{\nu}^{\frac{1}{3}+\epsilon}}{2\pi\abs{y_{\nu}}} \asymp \frac{T_{\nu}}{y_{\nu}}$ and recall the ideal $\imath$ from \eqref{eq:def_of_imath}. Further  fix $a\in \imath$ such that
\begin{equation} 
	\No(\imath) \leq \No((a)) \leq \left( \frac{2}{\pi} \right)^{r_1} \sqrt{\abs{d_F}}\No(\imath). \label{eq:def_of_a}
\end{equation}
This is possible by \cite[Lemma~6.2]{Ne13}. In particular one has $a\imath^{-1} \subset \mathcal{O}_F$.

Define
\begin{equation}
	I_{\nu}(k_{\nu}) = \begin{cases}
		\{ \xi_{\nu}\in F_{\nu}^{\times} \colon k_{\nu} \abs{a} R_{\nu} < \abs{\xi_{\nu}} \leq (k_{\nu}+1)\abs{a}R_{\nu} \}& \text{ if } k_{\nu} \geq 1, \\
		\{ \xi_{\nu}\in F_{\nu}^{\times} \colon \abs{\xi_{\nu}} \leq \abs{a}R_{\nu}, -k_{\nu} \leq \abs{\abs{\xi_{\nu}}-\frac{\abs{a}T_{\nu}}{2\pi\abs{y_{\nu}}}}< -k_{\nu}+1 \}& \text{else.}
	\end{cases} \label{eq:def_of_arch_box}
\end{equation}
For $\du{k}\in \mathbb{Z}^{\sharp\{\nu\}}$ let $I(\du{k}) = \prod_{\nu} I_{\nu}(k_{\nu})$. 

Let us start by establishing a simple but crucial property of these sets.

\begin{lemma}
If $k_{\nu}< -\lfloor \abs{a}R_{\nu} \rfloor$ then $I_{\nu}(k_{\nu}) = \emptyset$. \nonumber
\end{lemma}
\begin{proof}
Suppose $k_{\nu}< -\lfloor \abs{a}R_{\nu} \rfloor$. We consider two cases. First let $\abs{\xi_{\nu}}>\frac{\abs{a}T_{\nu}}{2\pi\abs{y_{\nu}}}$. Then the two inequalities in the definition of $I_{\nu}(\cdot)$ yield
\begin{equation}
	\frac{\abs{a}T_{\nu}}{2\pi\abs{y_{\nu}}}+\lfloor \abs{a}R_{\nu} \rfloor < \abs{\xi_{\nu}} \leq \abs{a}R_{\nu}. \nonumber
\end{equation}
But the set of such $\xi_{\nu}$ is empty. 
Secondly we assume $\abs{\xi_{\nu}} \leq \frac{\abs{a}R_{\nu}}{2\pi}$. This gives
\begin{equation}
	\abs{\xi_{\nu}} < \frac{\abs{a} T_{\nu}}{2\pi\abs{y_{\nu}}}- \lfloor \abs{a}R_{\nu} \rfloor<0 \nonumber
\end{equation}
which is also impossible.
\end{proof}

We also need good estimates for $\sharp(I(\du{k})\cap a\imath^{-1})$. These are obtained by a standard volume argument. Let us start with some pre-requests.

Choose a fundamental set $\mathcal{P}$ for the lattice $a\imath^{-1} \subset F_{\infty}$. Without loss of generality we can assume $0\in \mathcal{P}$. Let $D$ be the diameter of $\mathcal{P}$. It is an elementary fact (see \cite{Ne13}) that
\begin{equation}
	\vol(\mathcal{P})  \sim_F \No((a))\No(\imath^{-1}) \approx_F 1. \nonumber
\end{equation}

Further we define
\begin{equation}
	J_{\nu}(k_{\nu}) = \begin{cases}
			\{ \xi_{\nu}\in F_{\nu} \colon k_{\nu} \abs{a} R_{\nu}-D < \abs{\xi_{\nu}} \leq (k_{\nu}+1)\abs{a}R_{\nu}+D \}& \text{ if } k_{\nu} \geq 1, \\
			\{ \xi_{\nu}\in F_{\nu} \colon -k_{\nu}-D \leq \abs{\abs{\xi_{\nu}}-\frac{\abs{a}T_{\nu}}{2\pi\abs{y_{\nu}}}}< -k_{\nu}+1+D \}& \text{else.}
		\end{cases} \nonumber
\end{equation} 
and $J(\du{k}) = \prod_{\nu}J_{\nu}(k_{\nu})$. 

\begin{lemma}
The volume of $J_{\nu}(k_{\nu})$ is given by
\begin{equation}
	\vol(J_{\nu}(k_{\nu})) = \begin{cases}
		2\abs{a}R_{\nu}+4D & \text{ if $\nu$ is real and $k_{\nu}\geq 1$,} \\
		4(1+2D)& \text{ if $\nu$ is real and $k_{\nu}\leq 0$,} \\
		\pi(2k_{\nu}+1)\abs{a}R_{\nu}(\abs{a}R_{\nu}+2D)& \text{ if $\nu$ is complex and $k_{\nu}\geq 1$,} \\
		2\frac{\abs{a}T_{\nu}}{y_{\nu}}(1+2D)& \text{ if $\nu$ is complex and $k_{\nu}\leq 0$.}
	\end{cases} \nonumber
\end{equation}
\end{lemma}
\begin{proof}
	The proof is an elementary volume calculation. Its much easier when one draws a picture.
\end{proof}
As consequence of Minkowski-theory we can choose $\mathcal{P}$ such that $D\ll \No(a\imath^{-1})^{\frac{1}{n}} \ll_F 1$. Therefore, it is clear that $\vol(J(\du{k})) \ll_F \prod_{\nu}f_{\nu}(k_{\nu})$ for
\begin{equation}
	f_{\nu}(k_{\nu}) = \begin{cases}
			\frac{\abs{a}T_{\nu}}{y_{\nu}}+1 & \text{ if $\nu$ is real and $k_{\nu}\geq 1$,} \\
			1 & \text{ if $\nu$ is real and $k_{\nu}\leq 0$,} \\
			k_{\nu}(\frac{\abs{a}T_{\nu}}{y_{\nu}}+1)^2& \text{ if $\nu$ is complex and $k_{\nu}\geq 1$,} \\
			\frac{\abs{a}T_{\nu}}{y_{\nu}}+1& \text{ if $\nu$ is complex and $k_{\nu}\leq 0$.}
		\end{cases} \label{eq:def_of_f}
\end{equation}

We are now ready to count points in our boxes.
\begin{lemma} \label{lm:box_count_arch}
One has
\begin{equation}
	\sharp(a\imath^{-1}\cap I(\du{k})) \ll_{F} \prod_{\nu} f_{\nu}(k_{\nu}). \nonumber
\end{equation}
\end{lemma} 
\begin{proof}
By the construction of $\mathcal{P}$ we have
\begin{equation}
	\sharp(a\imath^{-1}\cap I(\du{k})) = \frac{\vol(\bigcup_{q\in a\imath^{-1}\cap I(\du{k})}(q+\mathcal{P}))}{\vol(\mathcal{P})} \leq \frac{\vol(J(\du{k}))}{\vol(\mathcal{P})}. \nonumber
\end{equation}
We conclude by the calculations above.
\end{proof}

For the estimation of $S_2(R)$ we need to count field elements with strong non-archimedean restrictions. We will be able to reduce this problem to \cite[Lemma~6]{BHMM16}. 

Define the sets 
\begin{eqnarray}
	\Z^{\n} &=& \prod_{\p\vert \n} \{k_{\p}\in\Z \colon k_{\p} \geq -v_{\p}(\imath) \}, \nonumber \\
	\A_{fin}^{\imath} &=& \{ a\in \A_{fin} \colon v_{\p}(a_{\p}) \geq -v_{\p}(\imath) \}, \nonumber \\
	C^{\imath}(\du{k}) &=& \{ a\in \A_{fin}^{\imath}  \colon v_{\p}(a_{\p}) = k_{\p} \quad \forall \p\vert \n \} \nonumber\\
	C^{\imath}(\du{k},[\du{u}]) &=& \{ a\in C^{\imath}(\du{k})  \colon a_{\p} = \varpi_{\p}^{k_{\p}}a_{\p}' \text{ with } [a'_{\p}] = [u_{\p}] \in \op_{\p}^{\times}/(1+\varpi_{\p}^{n_{0,\p}(g_{\p})}\op_{\p})  \quad \forall \p\vert \n\}. \nonumber
\end{eqnarray} 

It will be useful to know the volumes of these sets.
\begin{lemma}
We have
\begin{eqnarray}
	\vol(\A_{fin}^{\imath},d\mu_{fin}) &=& \No(\imath),\nonumber \\
	\vol(C^{\imath}(\du{k}),d\mu_{fin}) &=& \frac{\No(\imath)}{\No([\imath]_{\n})}\zeta_{\n}(1)\prod_{\p\vert\n} q_{\p}^{-k_{\p}}\nonumber \\
	\vol(C^{\imath}(\du{k},[\du{u}])d\mu_{fin}) &=& \vol(C^{\imath}(\du{k},[\du{u'}])d\mu), \nonumber \\
	\vol(C^{\imath}(\du{k},[\du{u}]),d\mu_{fin}) &=& \frac{\No(\imath)}{\No([\imath]_{\n})} \prod_{\p\vert\n} q_{\p}^{-k_{\p}-n_{0,\p}(g_{\p})}.\label{eq:vol_of_strange} 
\end{eqnarray}
\end{lemma}
\begin{proof}
This is a standard ad\'elic volume computation done place by place. The key facts we use are $\mu_{\p}(\op_{\p}^{\times}) =\zeta_{\p}(1)^{-1}$, $\mu_{\p}(\varpi_{\p}^{r}\op_{\p}) = q_{\p}^{-r}$, and that both $\mu_{\p}$ and $\mu_{\p}^{\times}$ are Haar measures for $\op^{\times}_{\p}$.
\end{proof}

Finally we are ready to prove the following counting result.
\begin{lemma} \label{lm:cruc_count}
We have
\begin{equation}
	\sharp\big((a\imath^{-1}\setminus\{0\})\cap B(R)C^{\imath}(\du{k},[\du{u}])\big) \ll F_{R}(\du{k}) = 1+\frac{\abs{R}_{\infty}\No(\imath)}{\No(\n_0(g))\No([\imath]_{\n})}\prod_{\p\vert\n }q_{\p}^{-k_{\p}} \nonumber
\end{equation}
uniform in $[\du{u}]$. Furthermore for $\prod_{\p\vert \n} q_{\p}^{k_{p}} > \abs{R}_{\infty}\No(\imath^{-1}[\imath]_{\n})$ there is no $q\in \big(F^{\times}\cap B(R)C^{\imath}(\du{k})\big)\setminus \{0\}$. 
\end{lemma}
\begin{proof}
Let us call the set we want to count $S$. If $S$ is empty we have nothing to show. Thus, take $q_0\in S$. Now define the shifted set $S' = \frac{1}{q_0}S-1$. Any $x\in S'$ satisfies
\begin{eqnarray}
	\abs{x}_{\nu} &\leq& 2\abs{\frac{R}{q_0}}_{\nu} \text{ for all } \nu,\nonumber \\
	\abs{x}_{\p} &\leq& \abs{\frac{\varpi_{\p}^{-v_{\p}(\imath)}}{q_0}}_{\p} \text{ for all } \p \nmid \n \text{ and} \nonumber\\
	\abs{x}_{\p} &\leq& \abs{\varpi_{\p}^{n_{0,\p}(g_{\p})}}_{\p}  \text{ for all } \p \mid \n. \nonumber
\end{eqnarray} 
Define the id\'ele $s$ by $s_{\nu} = 2^{1/[F_{\nu}:\R]}\frac{R}{q_0}$ and
\begin{equation}
	s_{\p} = \begin{cases} \frac{\varpi_{\p}^{-v_{\p}(\imath)}}{q_0} &\text{ if } \p\not\vert \n, \\  \varpi_{\p}^{n_{0,\p}(g_{\p})} &\text{else.}\end{cases} \nonumber
\end{equation}
After noting that $0\in S'$ we conclude that
\begin{equation}
	\sharp S \leq 1+ \sharp \{x\in F^{\times} \vert \abs{x}_{\nu}\leq \abs{s}_{\nu} \text{ and }\abs{x}_{\p}\leq \abs{s}_{\p} \}. \nonumber
\end{equation}
To estimate the last set we use \cite[Lemma~6]{BHMM16}. This yields
\begin{equation}
	\sharp S \leq 1+\abs{s}_{\A_F}. \nonumber
\end{equation}
We are left with calculating the ad\'elic norm of $s$. This is done using
\begin{equation}
	\prod_{\nu} \abs{q_0}^{-1}_{\nu} \prod_{\p\nmid \n} \abs{q_0}_{\p}^{-1} = \prod_{\p\mid\n}\abs{q_0}_{\p} = \prod_{\p\mid\n} q_{\p}^{-k_{\p}}. \nonumber
\end{equation}

To prove the second part we suppose $\prod_{\p\mid \n} q_{\p}^{k_{p}} > \abs{R}_{\infty}\No(\imath^{-1}[\imath]_{\n})$. Let us define the ideal $\mathfrak{m} = \prod_{\p\mid \n} \p^{k_{p}}$. Then in order to have $q\in C^{\imath}(\du{k})$ one needs $\No((q))\geq \No(\mathfrak{m})\imath[\imath]_{\n}^{-1}$. But for $q\in B(R)$ we require $\abs{q}_{\infty} \leq \abs{R}_{\infty}$. Now we conclude by
\begin{equation}
	1 = \abs{q}_{\A} = \abs{q}_{\infty} \abs{q}_{fin} = \frac{\abs{q}_{\infty}}{\No((q))}\leq \frac{\abs{R}_{\infty}\No([\imath]_{\n})}{\No(\mathfrak{m}\imath)}<1. \nonumber
\end{equation}
\end{proof}

Roughly the same reasoning applies to elements of $\imath^{-1}\cap B(R)$. 

\begin{cor} \label{cor:when_B_empty}
If $\abs{R}_{\infty} < \No(\imath)^{-1}$ then
\begin{equation}
	\imath^{-1}\cap B(R) = \{0\}. \nonumber
\end{equation}
\end{cor}

\subsection{The sum $S_1$} \label{sec:est_of_S1}

In this section we will treat the sum $S_1$. Due to the transition region of the archimedean Whittaker function this argument requires 
\begin{equation}
	R_{\nu} = \frac{T_{\nu}+T_{\nu}^{\frac{1}{3}+\epsilon}}{2\pi\abs{y_{\nu}}} \asymp \frac{T_{\nu}}{y_{\nu}}. \label{eq:def_of_R}
\end{equation}
Note that in view of Corollary~\ref{cor:when_B_empty} the sum $S_1$ is empty if $\abs{R}_{\infty}< \No(\imath)^{-1}$. Therefore we assume 
\begin{equation}
	\abs{R}_{\infty} \asymp \abs{\frac{T}{y}}_{\infty} \gg_F \No(\n_0\mathfrak{m}_{1}(g))^{-1}  \nonumber
\end{equation}
throughout this section. 
Let us fix $a\in \imath$ such that \eqref{eq:def_of_a} holds.

\begin{lemma} \label{lm:bound_S1}
We have
\begin{equation}
		S_1 \ll_{F}\abs{y}_{\infty}^{\frac{1}{2}} \abs{T}_{\infty}^{-\frac{1}{2}} \prod_{\nu} \left(\abs{T_{\nu}}_{\nu}^{\frac{1}{6}}+\abs{a}_{\nu}^{\frac{1}{4}} \abs{\frac{T_{\nu}}{y_{\nu}}}_{\nu}^{\frac{1}{4}} \right)^{1+\epsilon}. \nonumber
\end{equation}
\end{lemma}

The proof will be in the spirit of \cite{BHMM16}.

We start by expressing the archimedean Whittaker functions explicitly in terms of the $K$-Bessel function. We have
\begin{equation}
	\abs{W_{\nu}(a(\xi_{\nu}))} =  \begin{cases}
		\frac{\abs{\xi_{\nu}}_{\nu}^{\frac{1}{2}}\abs{K_{it_{\nu}}(2\pi \abs{\xi_{\nu}})}}{\abs{\Gamma(\frac{1}{2}+it_{\nu})\Gamma(\frac{1}{2}-it_{\nu})}^{\frac{1}{2}}} \text{ if $\nu$ is real, } \\
		\frac{\abs{\xi_{\nu}}_{\nu}^{\frac{1}{2}}\abs{K_{i2t_{\nu}}(4\pi \abs{\xi_{\nu}})}}{\abs{\Gamma(1+i2t_{\nu})\Gamma(1-i2t_{\nu})}^{\frac{1}{2}}} \text{ else.}
	\end{cases} \nonumber
\end{equation}
This holds as in \cite[p. 19]{BHMM16}. One notes, that the Gamma factors are due to the $L^2$-normalization in the archimedean Whittaker model. 

By Stirling's approximation one finds $\abs{\Gamma(\frac{1}{2}+it_{\nu})} \gg e^{-\frac{\pi}{2}t_{\nu}}$ and $\abs{\Gamma(1+2it)} \gg T_{\nu}^{\frac{1}{2}}e^{-\pi t_{\nu}}$. Thus using \cite[(3.1)]{Te15} one derives
\begin{equation}
	W_{\nu}(a(qy)) \ll \abs{qy}_{\nu}^{\frac{1}{2}} \frac{T_{\nu}^{\frac{1}{2}}}{\abs{T_{\nu}}_{\nu}^{\frac{1}{2}}}\min\left(T_{\nu}^{-\frac{1}{3}}, (t_{\nu}\abs{y_{\nu}})^{-\frac{1}{4}}\abs{\abs{q}-\frac{T_{\nu}}{2\pi\abs{y_{\nu}}}}^{-\frac{1}{4}}\right).  \nonumber
\end{equation}

We define
\begin{equation}
	g_{\nu} (k_{\nu}) = \min\left(T_{\nu}^{\frac{1}{6}}, \abs{\frac{aT_{\nu}}{y_{\nu}k_{\nu}}}^{\frac{1}{4}}\right) \nonumber
\end{equation}
and observe that for $k_{\nu}\leq 0$ and $q\in I_{\nu}(k_{\nu})$ we have
\begin{equation}
	  W_{\nu}(a(qa^{-1}y)) \ll \abs{qa^{-1}y}_{\nu}^{\frac{1}{2}} \abs{T_{\nu}}_{\nu}^{-\frac{1}{2}} g_{\nu}(k_{\nu}).  \label{eq:bound_whitt_arch}
\end{equation}
But for $q\in B(R)$ we also have $\abs{qa^{-1}}_{\nu} \leq \abs{R}_{\nu} \asymp \abs{\frac{T_{\nu}}{y_{\nu}}}_{\nu}$ and hence $\abs{qa^{-1}y_{\nu}}_{\nu} \ll \abs{T_{\nu}}_{\nu}$. And thus
\begin{equation}
	W_{\nu}(a(qa^{-1}y_{\nu})) \ll   g_{\nu}(k_{\nu}).  \nonumber
\end{equation}

Finally we are ready to estimate $S_1$.
\begin{proof}[Proof of Lemma~\ref{lm:bound_S1}]
First we shift the sum by $a$. This gives
\begin{equation}
	S_1^4 = \abs{a}_{\infty}^2\sum_{q\in a\imath^{-1}\cap B(\abs{a}R)}  \abs{q}_{\infty}^{-2}\abs{W_{\infty}(a(qa^{-1}y))}^4. \nonumber
\end{equation}
Then we partition $B(\abs{a}R)$ using the boxes defined in \eqref{eq:def_of_arch_box}. In each box we exploit \eqref{eq:bound_whitt_arch} to get
\begin{equation}
	S_1^4 \ll \abs{a}_{\infty}^2\sum_{\substack{\du{k}\in\Z^{\sharp\{\nu\}}, \\ -\lfloor aR_{\nu} \rfloor\leq k_{\nu} \leq 0 }} \sharp(I(\du{k})\cap a\imath^{-1})\prod_{\nu}\abs{a^{-1} y_{\nu}}_{\nu}^2 \abs{T_{\nu}}_{\nu}^{-2}   g_{\nu}(k_{\nu})^4. \nonumber
\end{equation}
Inserting the result from Lemma~\ref{lm:box_count_arch} yields
\begin{equation}
	S_1^4 \ll \abs{y}_{\infty}^2 \abs{T}_{\infty}^{-2} \prod_{\nu} \sum_{k_{\nu}=0}^{\lfloor \abs{a}R_{\nu}\rfloor} g_{\nu}(-k_{\nu})^4f_{\nu}(-k_{\nu}). \nonumber
\end{equation}
To estimate the remaining sums we use ideas of \cite{BHMM16}. 

Let us treat each place at its own. We start with $\nu$ real and obtain
\begin{eqnarray}
	\sum_{k_{\nu}=0}^{\lfloor \abs{a}R_{\nu}\rfloor} g_{\nu}(-k_{\nu})^4f_{\nu}(-k_{\nu}) &=& T_{\nu}^{\frac{2}{3}} + \sum_{k_{\nu}=1}^{\lfloor \abs{a}R_{\nu}\rfloor} \frac{\abs{a}T_{\nu}}{\abs{y_{\nu}}k_{\nu}} \nonumber\\
	&\ll& \left( \abs{T_{\nu}}_{\nu}^{\frac{2}{3}}+\frac{\abs{a}T_{\nu}}{\abs{y_{\nu}}}\right)^{1+\epsilon}.\nonumber
\end{eqnarray}
Similarly one treats the complex places:
\begin{eqnarray}
	\sum_{k_{\nu}=0}^{\lfloor \abs{a}R_{\nu}\rfloor} g_{\nu}(-k_{\nu})^4f_{\nu}(-k_{\nu}) &\leq& T_{\nu}^{\frac{2}{3}}\left(\abs{a}\frac{T_{\nu}}{\abs{y_{\nu}}}+1\right) + \sum_{k_{\nu}=1}^{\lfloor \abs{a}R_{\nu}\rfloor} \left(\abs{a}\frac{T_{\nu}}{\abs{y_{\nu}}}+1\right)\frac{\abs{a}T_{\nu}}{\abs{y_{\nu}}k_{\nu}} \nonumber \\
	&\ll& \left(\abs{a}\frac{T_{\nu}}{\abs{y_{\nu}}}+1\right)\left(T_{\nu}^{\frac{2}{3}}+\abs{a}\frac{T_{\nu}}{\abs{y_{\nu}}}\right)^{1+\epsilon} \ll \left(T_{\nu}^{\frac{2}{3}}+\abs{a}\frac{T_{\nu}}{\abs{y_{\nu}}}\right)^{2+\epsilon} \nonumber \\
	&\ll& \left(T_{\nu}^{\frac{4}{3}}+\abs{a}^{2}\frac{T_{\nu}^2}{ y_{\nu}^2}\right)^{1+\epsilon}. \nonumber 
\end{eqnarray}

Putting everything together gives
\begin{equation}
	S_1\ll_{F}\abs{y}_{\infty}^{\frac{1}{2}} \abs{T}_{\infty}^{-\frac{1}{2}}  \prod_{\nu} \left(\abs{T_{\nu}}_{\nu}^{\frac{1}{6}}+\abs{a}_{\nu}^{\frac{1}{4}} \abs{\frac{T_{\nu}}{y_{\nu}}}_{\nu}^{\frac{1}{4}} \right)^{1+\epsilon}. \nonumber
\end{equation}
\end{proof}

\begin{cor} \label{cor:after_balancing}
If we assume
\begin{equation}
	\abs{y_{\nu}}_{\nu} \asymp \abs{a^3T_{\nu}}_{\nu}^{\frac{\log(\abs{y}_{\infty})}{\log(\abs{a^3T}_{\infty})}} \label{eq:balancedness_of_y}
\end{equation}
for all $\nu$ then we obtain
\begin{equation}
	S_1 \ll \abs{y}_{\infty}^{\frac{1}{2}} \abs{T}_{\infty}^{-\frac{1}{2}} \left( \abs{T}_{\infty}^{\frac{1}{6}}+\No(\n_0\mathfrak{m}_1(g))^{\frac{1}{4}} \abs{\frac{T}{y}}_{\infty}^{\frac{1}{4}} \right)^{1+\epsilon}. \nonumber
\end{equation}
\end{cor}
\begin{proof}
We consider two cases. First assume $\abs{y}_{\infty} \leq \abs{a^3T}_{\infty}^{\frac{1}{3}}$. Then the balancing assumption implies $\abs{y_{\nu}}_{\nu} \ll\abs{a}_{\nu}\abs{T_{\nu}}_{\nu}^{\frac{1}{3}}$ for all $\nu$. Therefore, we have
\begin{equation}
	\abs{\frac{aT_{\nu}}{y_{\nu}}}_{\nu}^{\frac{1}{4}} \gg \abs{T_{\nu}}_{\nu}^{\frac{1}{6}}. \nonumber
\end{equation}

Secondly, if $\abs{y}_{\infty} \geq \abs{a^3T}_{\infty}^{\frac{1}{3}}$, one argues analogously to obtain
\begin{equation}
	\abs{\frac{aT_{\nu}}{y_{\nu}}}_{\nu}^{\frac{1}{4}} \ll \abs{T_{\nu}}_{\nu}^{\frac{1}{6}}. \nonumber
\end{equation}

Recalling that $\abs{a}_{\infty} \ll_F \No(\n_0\mathfrak{m}_1(g))$ completes the proof.
\end{proof}

\subsection{The sum $S_2(R)$} \label{sec:est_S2}

In this section we will estimate the sum $S_2(R)$ by reducing it to well known averages of Hecke eigenvalues and local Whittaker functions.

\begin{lemma} \label{lm:bound_for_S2}
We have
\begin{equation}
	S_2(R) \ll \nonumber  (\abs{T}_{\infty}\No(\n))^{\epsilon} \abs{R}_{\infty}^{\frac{1}{4}+\epsilon}\left( \frac{\No(\n_0)^{\frac{1}{4}}}{\No(\mathfrak{m}_1(g))^{\frac{1}{4}}}+\abs{R}_{\infty}^{\frac{1}{2}}\No(\n_0\mathfrak{m}_1(g))^{\frac{1}{4}} \right). \nonumber
\end{equation}
\end{lemma}

\begin{proof}
We start by defining
\begin{equation}
	I(\mathfrak{m}) = \{ q\in \imath^{-1}\vert \abs{q}_{\nu} \leq \abs{R_{\nu}}_{\nu}, (q)=\mathfrak{m} \}. \nonumber
\end{equation}
Using \cite[Corollary~1]{BHMM16} we observe that
\begin{equation}
	\sharp I(\mathfrak{m}) \ll_{\epsilon} \abs{R}_{\infty}^{\epsilon}\No(\mathfrak{m})^{-\epsilon}. \label{eq:shortness_q_sum}
\end{equation}
In particular, if $\No(\mathfrak{m}) \gg \abs{R}_{\infty}$  then $I(\mathfrak{m})$ must be empty.

By Lemma~\ref{lm:unram_rel} we have
\begin{eqnarray}
	S_2(R)^{\frac{4}{3}} &=& \sum_{\substack{\mathfrak{m}\subset \imath^{-1}, \\ \No(\mathfrak{m})\ll \abs{R}_{\infty}}} \No(\mathfrak{m})^{\frac{2}{3}}\No([\mathfrak{m}\imath]_{\n})^{\frac{4}{3}}\frac{\abs{\lambda\left(\frac{\mathfrak{m}\imath}{[\mathfrak{m}\imath]_{\n}}\right)}^{\frac{4}{3}}}{\No(\mathfrak{m}\imath)^{\frac{4}{3}}} \sum_{q\in I(\mathfrak{m})} \abs{\lambda_{\n}(q)}^{\frac{4}{3}} \nonumber \\
	&=& \No(\imath)^{-\frac{2}{3}} \sum_{\substack{\mathfrak{m_1}\vert \n^{\infty}, \\ \No(\mathfrak{m_1})\ll \No(\imath)\abs{R}_{\infty}}} \No(\mathfrak{m}_1)^{\frac{2}{3}}\sum_{\substack{(\mathfrak{m_2},\n)=1, \\ \No(\mathfrak{m_2})\ll \frac{\No(\imath)\abs{R}_{\infty}}{\No(\mathfrak{m}_1)}}} \frac{\abs{\lambda(\mathfrak{m}_2)}^{\frac{4}{3}}}{\No(\mathfrak{m}_2)^{\frac{2}{3}}} \sum_{q\in I(\imath^{-1}\mathfrak{m}_1\mathfrak{m_2})}\abs{\lambda_{\n}(q)}^{\frac{4}{3}}. \nonumber
\end{eqnarray}

At this stage we apply H\"older to the $\mathfrak{m}_2$-sum. This yields
\begin{eqnarray}
	S_2(R)^{\frac{4}{3}} &=& \No(\imath)^{-\frac{2}{3}} \sum_{\substack{\mathfrak{m_1}\vert \n^{\infty}, \\ \No(\mathfrak{m_1})\ll \No(\imath)\abs{R}_{\infty}}} \No(\mathfrak{m}_1)^{\frac{2}{3}} S_{\text{ur}}\left(\frac{\No(\imath)\abs{R}_{\infty}}{\No(\mathfrak{m}_1)}\right)^{\frac{1}{3}} \nonumber\\
	&&\qquad\qquad \cdot\left(\sum_{\substack{(\mathfrak{m_2},\n)=1, \\ \No(\mathfrak{m_2})\ll \frac{\No(\imath)\abs{R}_{\infty}}{\No(\mathfrak{m}_1)}}} \left( \sum_{q\in I(\imath^{-1}\mathfrak{m}_1\mathfrak{m_2})}\abs{\lambda_{\n}(q)}^{\frac{4}{3}} \right)^{\frac{3}{2}} \right)^{\frac{2}{3}}.  \nonumber
\end{eqnarray}
Here
\begin{equation}
	S_{\text{ur}}(X) = \sum_{\substack{(\mathfrak{m},\n)=1, \\ \No(\mathfrak{m})\leq X}} \frac{\abs{\lambda(\mathfrak{m})}^4}{\No(\mathfrak{m})^2}. \nonumber
\end{equation}
Before we continue it is important to recall that our Hecke operators are differently normalized than the ones in \cite{BHMM16}, \cite{Sa15} and \cite{HL94}. It is well known that
\begin{equation}
	S_{\text{ur}}(X) \ll_{F,\epsilon} (\abs{T}_{\infty} \No(\n))^{\epsilon} X^{1+\epsilon}. \nonumber
\end{equation}
This was proved in \cite{HL94} over $\Q$. 

We now use Jensen's inequality exploiting that the $q$-sum is short by \eqref{eq:shortness_q_sum}. This yields
\begin{eqnarray}
	S_2(R)^{\frac{4}{3}} &\ll& (\abs{T}_{\infty}\abs{R}_{\infty}\No(\n))^{\epsilon} \No(\imath)^{-\frac{1}{3}}\abs{R}_{\infty}^{\frac{1}{3}} \nonumber \\
	&&\cdot\sum_{\substack{\mathfrak{m_1}\vert \n^{\infty}, \\ \No(\mathfrak{m_1})\ll \No(\imath)\abs{R}_{\infty}}} \No(\mathfrak{m}_1)^{\frac{1}{3}+\epsilon} \left(\sum_{\substack{(\mathfrak{m_2},\n)=1, \\ \No(\mathfrak{m_2})\ll \frac{\No(\imath)\abs{R}_{\infty}}{\No(\mathfrak{m}_1)}}} \sum_{q\in I(\imath^{-1}\mathfrak{m}_1\mathfrak{m_2})}\abs{\lambda_{\n}(q)}^2\right)^{\frac{2}{3}}. \label{eq:after_jensen}
\end{eqnarray}

We will first continue to analyze the $\mathfrak{m}_2$-sum. For the sake of notation we define
\begin{equation}
	S_{\text{ram}} =\sum_{\substack{(\mathfrak{m_2},\n)=1, \\ \No(\mathfrak{m_2})\ll \frac{\No(\imath)\abs{R}_{\infty}}{\No(\mathfrak{m}_1)}}} \sum_{q\in I(\imath^{-1}\mathfrak{m}_1\mathfrak{m_2})}\abs{\lambda_{\n}(q)}^2
\end{equation}

In order to use the notation from Section~\ref{sec:counting_in_boxes} we set
\begin{equation}
	\du{k}(\mathfrak{m}) = (v_{\p}(\mathfrak{m}))_{\p\vert\n}. \nonumber 
\end{equation}

By the local definition of  $\lambda_{\n}$ we can view it as a function on $\A_{fin}^{\imath}$. Lemma~\ref{lm:constans} then implies that this function is constant on the sets $C^{\imath}(\du{k},[\du{u}])$. Therefore, we have
\begin{eqnarray}
	&&S_{\text{ram}} =  \sum_{q\in \imath^{-1}\cap B(R)C^{\imath}(\du{k}(\mathfrak{m_1}\imath^{-1}))} \abs{\lambda_{\n}(q)}^2 \nonumber \\
	&&=\sum_{[\du{u}]\in \prod_{\p\vert\n}\op_{\p}^{\times}/(1+\varpi_{\p}^{n_{0,\p}}\op_{\p})} \sum_{q\in \imath^{-1}\cap B(R)C^{\imath}(\du{k}(\mathfrak{m_1}\imath^{-1}),[\du{u}])} \abs{\lambda_{\n}(q)}^2 \nonumber \\
	&&= \sum_{[\du{u}]\in \prod_{\p\vert\n}\op_{\p}^{\times}/(1+\varpi_{\p}^{n_{0,\p}}\op_{\p})} \frac{\sharp(\imath^{-1}\cap B(R)C^{\imath}(\du{k}(\mathfrak{m_1}\imath^{-1}),[\du{u}]))}{\vol(C^{\imath}(\du{k}(\mathfrak{m_1}\imath^{-1}),[\du{u}]),d\mu)} \int_{C^{\imath}(\du{k}(\mathfrak{m_1}\imath^{-1}),[\du{u}])} \abs{\lambda_{\n}(q)}^2 d\mu_{\text{fin}}(q). \nonumber
\end{eqnarray}

Using \eqref{eq:vol_of_strange} and Lemma~\ref{lm:cruc_count} reveals
\begin{equation}
	S_{\text{ram}} \ll \frac{\No(\mathfrak{m}_1)\No(\n_0(g))}{\No(\imath)}  F_{R}(\du{k}(\mathfrak{m_1}\imath^{-1}))\int_{C^{\imath}(\du{k}(\mathfrak{m_1}\imath^{-1}))} \abs{\lambda_{\n}(q)}^2d\mu_{\text{fin}}(q). \label{eq:intermediate_ram_for_later}
\end{equation}

The integral appearing here can be estimated using the local result \cite[Proposition~2.11]{Sa15}. This is done as follows:
\begin{eqnarray}
	\int_{C^{\imath}(\du{k})} \abs{\lambda_{\n}(q)}^2 d\mu_{\text{fin}}(q) &=& \prod_{\p\nmid\n} \int_{\varpi_{\p}^{-v_{\p}(\imath)}\op_{\p}} 1 d\mu_{\p} \prod_{\p\mid\n} \int_{\varpi_{\p}^{k_{\p}}\op_{\p}^{\times}} \abs{W_{\p}(a(\varpi_{\p}^{v_{\p}(\mathfrak{d})}\theta_i q)g_{\p})}^2 d\mu_{\p}(q) \nonumber\\ 
	&=& \frac{\No(\imath)}{\No([\imath]_{\n})}\zeta_{\n}(1)^{-1} \prod_{\p\mid\n} q_{\p}^{-k_{\p}} \int_{\op_{\p}^{\times}}\abs{W_{\p}(a(\varpi_{\p}^{v_{\p}(\mathfrak{d})+v_{\p}(\theta_i)+k_{\p}}q)g_{\p})}^2d\mu_{\p}^{\times}(q) \nonumber \\
	&\ll&  \No(\n)^{\epsilon}\frac{\No(\imath)}{\No([\imath]_{\n})}\zeta_{\n}(1)^{-1}  \prod_{\p\mid \n} q_{\p}^{-\frac{1}{2}(v_{p}(\mathfrak{d})+v_{\p}(\theta_i)+n_0+m_1(g_{\p})+3k_{\p})} \nonumber \\
	&=& \No(\n)^{\epsilon}\frac{\No(\imath)}{\No([\imath]_{\n})^{\frac{3}{2}}}\zeta_{\n}(1)^{-1}  \prod_{\p\mid \n} q_{\p}^{-\frac{3k_{\p}}{2}}. \nonumber
\end{eqnarray} 
Note that here we crucially rely on $g_{\p}\in\mathfrak{J}_{\n}$ in order to apply the upper bounds for the local integrals.
Inserting this estimate in our expression for $S_{\text{ram}}$ we get
\begin{eqnarray}
	S_{\text{ram}} \ll \zeta_{\n}(1)^{-1}\No(\n_0(g))\No(\mathfrak{m}_1)^{-\frac{1}{2}} F_{R}(\du{k}(\mathfrak{m_1}\imath^{-1})). \nonumber
\end{eqnarray}
The result from Lemma~\ref{lm:cruc_count} yields
\begin{equation}
	S_{\text{ram}} \ll \zeta_{\n}(1)^{-1}\left( \frac{\No(\n_0(g))}{\No(\mathfrak{m}_1)^{\frac{1}{2}}}+\frac{\abs{R}_{\infty}\No(\imath)}{\No(\mathfrak{m}_1)^{\frac{3}{2}}} \right). \nonumber
\end{equation}

From \eqref{eq:after_jensen} we deduce
\begin{equation}
	S_2(R) \ll (\abs{T}_{\infty}\abs{R}_{\infty}\No(\n))^{\epsilon} \frac{\abs{R}_{\infty}^{\frac{1}{4}}}{\No(\imath)^{\frac{1}{4}}}\left( \sqrt{\No(\n_0(g))}+\sqrt{\abs{R}_{\infty}\No(\imath)} \right) \left( \sum_{\substack{\mathfrak{m_1}\vert \n^{\infty}, \\ \No(\mathfrak{m_1})\ll \No(\imath)\abs{R}_{\infty}}} \No(\mathfrak{m}_1)^{\epsilon} \right)^{\frac{4}{3}}. \nonumber
\end{equation}

In the end we note that by the Rankin-trick we have
\begin{equation}
	\sum_{\substack{\mathfrak{m_1}\vert \n^{\infty}, \\ \No(\mathfrak{m_1})\ll \No(\imath)\abs{R}_{\infty}}} \No(\mathfrak{m}_1)^{\epsilon} \ll \No(\n)^{\epsilon} \abs{R}_{\infty}^{\epsilon}. \nonumber
\end{equation}
\end{proof}

\subsection{The error $\mathcal{E}$}

For $R$ as in \eqref{eq:def_of_R} we will roughly prove that the error is always absorbed in the main-contribution. More precisely we have the following lemma.

\begin{lemma} \label{lm:est_error}
Under the balancing assumption
\begin{equation}
	\abs{y_{\nu}}_{\nu} \asymp \abs{a^3T_{\nu}}_{\nu}^{\frac{\log(\abs{y}_{\infty})}{\log(\abs{a^3T}_{\infty})}} \text{ for all } \nu \nonumber
\end{equation}
we have
\begin{eqnarray}
	\mathcal{E} &\ll& (\abs{T}_{\infty}\abs{R}_{\infty}\No(\n))^{\epsilon} \bigg( \abs{T}_{\infty}^{\frac{1}{6}}\No(\n_0)^{\frac{1}{2}}+\abs{T}_{\infty}^{\frac{1}{6}}\abs{R}_{\infty}^{\frac{1}{4}}\No(\n_0\mathfrak{m}_1(g))^{\frac{1}{4}}+\abs{R}_{\infty}^{\frac{1}{2}}\No(\n_0\mathfrak{m}_1(g))^{\frac{1}{2}} \bigg) \nonumber
\end{eqnarray}
\end{lemma}
\begin{proof}
For $S\subset \{\nu\}$ we define
\begin{eqnarray}
	R'(\du{k}) &=& \begin{cases}
		(k_{\nu}+1) R_{\nu}& \text{ if } \nu\in S, \\
		R_{\nu} &\text{ else,}
	\end{cases} \nonumber \\
	I_{S}(\du{k}) &=& \prod_{\nu\in S} I_{\nu}(k_{\nu}), \nonumber \\
	B_S(R) &=& \prod_{\nu\not\in S} \{ \xi_{\nu} \in F^{\times} \vert \abs{\xi_{\nu}} \leq R_{\nu} \}. \nonumber
\end{eqnarray}
For $k_{\nu} \geq 1$ we define
\begin{equation}
	g_{\nu}(k_{\nu}) = e^{-\pi \abs{y_{\nu}}R_{\nu}k_{\nu}}.\nonumber
\end{equation}
By the exponential decay of the $K$-Bessel function we have the bound
\begin{equation}
	\abs{a^{-1}q}_{\nu}^{-2}\abs{W_{\nu}(a(a^{-1}qy_{\nu}))}^4 \ll k_{\nu}^{-2}R_{\nu}^{-2} g_{\nu}(k_{\nu})^4. \label{eq:whitt_exp_dec}
\end{equation}
for $q\in I_{\nu}(k_{\nu})$ and $k_{\nu}\geq 1$. We now decompose $\mathcal{E}$ as follows 
\begin{eqnarray}
	\mathcal{E} \leq \sum_{\emptyset\neq S\subset \{ \nu \}}\sum_{\du{k}\in \N^{\sharp S}} \left(\sum_{q\in a\imath^{-1}\cap I_S(\du{k})\times B_{S}(\abs{a}R)} \abs{a^{-1}q}_{\infty}^{-2} \abs{W_{\infty}(a^{-1}qy)}^4 \right)^{\frac{1}{4}}S_2(R'(\du{k})). \nonumber
\end{eqnarray}
Again we included the shift by $a$ only in the archimedean part. Note that by Corollary~\ref{cor:when_B_empty} below the sum $S_2(R'(\du{k})) = 0$ if $ \prod_{\nu\in S}\abs{k_{\nu}+1}_{\nu}\abs{R}_{\infty} < \No(\imath)^{-1}$. We can add the condition $\prod_{\nu\in S}\abs{k_{\nu}+1}_{\nu} \geq \abs{R}_{\infty}^{-1} \No(\imath)^{-1}$ to the sum over $\du{k}$. 

First note that Lemma~\ref{lm:bound_for_S2}  is good enough to deal with the non-archimedean part of the sum. To deal with the archimedean part we use the same approach as in Section~\ref{sec:est_of_S1}. In particular with \eqref{eq:bound_whitt_arch} and \eqref{eq:whitt_exp_dec} we have
\begin{eqnarray}
	&&\sum_{q\in a\imath^{-1}\cap I_S(\du{k})\times B_{S}(\abs{a}R)} \abs{qa^{-1}}_{\infty}^{-2}\abs{W_{\infty}(a(qa^{-1}y))}^4 \nonumber \\
	&=& \sum_{\substack{\du{k^c},\\  -\lfloor \abs{a}R_{\nu}\rfloor \leq k_{\nu}\leq 0 \forall\nu\not\in S }}\sum_{q\in a\imath^{-1}\cap I(\du{k}\times \du{k}^c)} \abs{qa^{-1}}_{\infty}^{-2}\abs{W_{\infty}(a(qa^{-1}y))}^4 \nonumber \\ 
	&\ll& \abs{R}_{\infty}^{-2}\prod_{\nu\in S} \abs{k_{\nu}}_{\nu}^{-2}g_{\nu}(k_{\nu})^4f_{\nu}(k_{\nu}) \prod_{\nu \not\in S} \sum_{\substack{\du{k^c},\\  -\lfloor \abs{a}R_{\nu}\rfloor \leq k_{\nu}\leq 0 \forall\nu\ni S }} g_{\nu}(k_{\nu})^4f_{\nu}(k_{\nu}) \nonumber \\
	&\ll& \abs{R}_{\infty}^{-2}\prod_{\nu\in S} \abs{k_{\nu}}_{\nu}^{-2}g_{\nu}(k_{\nu})^4f_{\nu}(k_{\nu}) \prod_{\nu \not\in S} \left( \abs{T_{\nu}}_{\nu}^{\frac{2}{3}}+\abs{aR_{\nu}}_{\nu} \right)^{1+\epsilon}. \nonumber
\end{eqnarray}
We obtain
\begin{eqnarray}
	\mathcal{E}	&\ll& \bigg(\abs{\frac{T}{y}}_{\infty}\No(\n)\bigg)^{\epsilon} \sum_{\emptyset\neq S\subset \{ \nu \}} \sum_{\substack{\du{k}\in \N^{\sharp S}, \\ \prod_{\nu\in S}\abs{k_{\nu}+1}_{\nu} \geq \abs{R}_{\infty}^{-1} \No(\imath)^{-1} }} \left(\frac{\abs{R}_{\infty}^{-\frac{1}{4}}\No(\n_0)^{\frac{1}{4}}}{\No(\mathfrak{m}_1(g))^{\frac{1}{4}}}+\abs{R}_{\infty}^{\frac{1}{4}}\No(\n_0\mathfrak{m}_1(g))^{\frac{1}{4}}  \right)\nonumber \\
	&&\qquad\qquad\qquad\qquad\qquad\qquad\qquad\qquad\qquad\qquad \cdot\prod_{\nu\not\in S}\left( \abs{T_{\nu}}_{\nu}^{\frac{1}{6}}+\abs{aR_{\nu}}_{\nu}^{\frac{1}{4}}\right) \prod_{\nu\in S}g_{\nu}(k_{\nu})f_{\nu}(k_{\nu})^{\frac{1}{4}}. \nonumber
\end{eqnarray}
Inserting the definition of $f_{\nu}$ from \eqref{eq:def_of_f} and using the balancing assumption as in the proof of Corollary~\ref{cor:after_balancing} yields
\begin{eqnarray}
	\mathcal{E} &\ll& \sum_{\emptyset\neq S\subset \{ \nu \}}\sum_{\substack{\du{k}\in \N^{\sharp S}, \\ \prod_{\nu\in S}\abs{k_{\nu}+1}_{\nu} \geq \abs{R}_{\infty}^{-1} \No(\imath)^{-1} }} \bigg[\left(\frac{\abs{R}_{\infty}^{-\frac{1}{4}}\No(\n_0)^{\frac{1}{4}}}{\No(\mathfrak{m}_1(g))^{\frac{1}{4}}}+\abs{R}_{\infty}^{\frac{1}{4}}\No(\n_0\mathfrak{m}_1(g))^{\frac{1}{4}}  \right)^{1+\epsilon}\nonumber \\
	&&\qquad\qquad\qquad\cdot\left(\abs{T}_{\infty}^{\frac{1}{6}}+\abs{R}_{\infty}^{\frac{1}{4}}\No(\n_0\mathfrak{m}_1(g))^{\frac{1}{4}} \right)\prod_{\nu\in S}k_{\nu}^{[F_{\nu}:\R]-1}g_{\nu}(k_{\nu}) \bigg] \nonumber \\
	&\ll& (\abs{T}_{\infty}\abs{R}_{\infty}\No(\n))^{\epsilon} \sum_{\emptyset\neq S\subset \{ \nu \}}\sum_{\substack{\du{k}\in \N^{\sharp S}, \\ \prod_{\nu\in S}\abs{k_{\nu}+1}_{\nu} \geq \abs{R}_{\infty}^{-1} \No(\imath)^{-1} }} \bigg[\bigg(\abs{R}_{\infty}^{-\frac{1}{4}} \frac{\abs{T}_{\infty}^{\frac{1}{6}}\No(\n_0)^{\frac{1}{4}}}{\No(\mathfrak{m}_1(g))^{\frac{1}{4}}}+\No(\n_0)^{\frac{1}{2}}\nonumber \\
	&&\qquad\qquad\qquad+\abs{T}_{\infty}^{\frac{1}{6}}\abs{R}_{\infty}^{\frac{1}{4}}\No(\n_0\mathfrak{m}_1(g))^{\frac{1}{4}}+\abs{R}_{\infty}^{\frac{1}{2}}\No(\n_0\mathfrak{m}_1(g))^{\frac{1}{2}} \bigg)\prod_{\nu\in S}k_{\nu}^{[F_{\nu}:\R]-1}g_{\nu}(k_{\nu}) \bigg]. \nonumber
\end{eqnarray}
Now we use the condition in the $\du{k}$-sum to remove the $\abs{R}^{-\frac{1}{4}}$. We can now drop any unnecessary condition on $\du{k}$ and end up with
\begin{eqnarray}
	\mathcal{E} &\ll& (\abs{T}_{\infty}\abs{R}_{\infty}\No(\n))^{\epsilon} \bigg( \abs{T}_{\infty}^{\frac{1}{6}}\No(\n_0)^{\frac{1}{2}}+\abs{T}_{\infty}^{\frac{1}{6}}\abs{R}_{\infty}^{\frac{1}{4}}\No(\n_0\mathfrak{m}_1(g))^{\frac{1}{4}}+\abs{R}_{\infty}^{\frac{1}{2}}\No(\n_0\mathfrak{m}_1(g))^{\frac{1}{2}} \bigg) \nonumber \\
	&&\qquad\qquad \sum_{\emptyset\neq S\subset \{ \nu \}}\sum_{\du{k}\in \N^{\sharp S}}\abs{k_{\nu}}^{\frac{[F_{\nu}:\R]}{2}-\frac{1}{4}}g_{\nu}(k_{\nu}).\nonumber
\end{eqnarray}
Due to the exponential decay of $g_{\nu}(k_{\nu})$ for positive $k_{\nu}$ it is no problem to estimate the remaining sums.
\end{proof}

\subsection{The final Whittaker bound}

Now we have all the pieces together to prove an upper bound for $\phi_{\circ}$ via its Whittaker expansion.

\begin{prop} \label{pr:whit_exp_est_bad}
Let $\phi_{\circ} = \sigma(v^{\circ})$ for some cuspidal automorphic representation $(\pi,V_{\pi})$ with new vector $v^{\circ}$. For $g\in \mathfrak{J}_{\n}$  we have
\begin{eqnarray}
	&&\abs{\phi_{\circ}(a(\theta_i)gn(x)a(y))} \nonumber \\
	&&\quad\ll_{F,\epsilon} (\abs{T}_{\infty}\abs{y}_{\infty}^{-1}\No(\n))^{\epsilon} \bigg( \abs{T}_{\infty}^{\frac{1}{6}}\No(\n_0)^{\frac{1}{2}}+\abs{T}_{\infty}^{\frac{1}{6}}\abs{R}_{\infty}^{\frac{1}{4}}\No(\n_0\mathfrak{m}_1(g))^{\frac{1}{4}}+\abs{R}_{\infty}^{\frac{1}{2}}\No(\n_0\mathfrak{m}_1(g))^{\frac{1}{2}} \bigg). \nonumber
\end{eqnarray}
\end{prop}
\begin{proof}
As in \cite[(8.7)]{BHMM16} we can assume that $y$ is balanced in the sense of \eqref{eq:balancedness_of_y}. 

Then we note that if $\abs{R}_{\infty}< \No(\imath)^{-1}$, it follows from Corollary~\ref{cor:when_B_empty} and \eqref{eq:whitt_after_ineq} that
\begin{equation}
	\abs{\psi(a(\theta_i)gn(x)a(y))} \leq \abs{c_{\phi}}\mathcal{E}. \nonumber
\end{equation}
In this case we get the desired bound from Lemma~\ref{lm:est_error} and \ref{lm:the_scaling_const}.

Thus, if $\abs{R}_{\infty} \geq \No(\imath)^{-1}$ the main contribution will obviously come from $S_1S_2(R)$. As in \cite[(8.7)]{BHMM16} we can assume that $y$ is balanced in the sense of \eqref{eq:balancedness_of_y} so that we can use Corollary~\ref{cor:after_balancing} and Lemma~\ref{lm:bound_for_S2} to show
\begin{eqnarray}
	S_1S_2(R) &\ll_{F,\epsilon}&  (\abs{\frac{T}{y}}_{\infty}\No(\n))^{\epsilon} \abs{R}_{\infty}^{-\frac{1}{4}} \left( \abs{T}_{\infty}^{\frac{1}{6}}+\No(\n_0\mathfrak{m}_1(g))^{\frac{1}{4}} \abs{\frac{T}{y}}_{\infty}^{\frac{1}{4}} \right)^{1+\epsilon} \nonumber \\
	&&\qquad\qquad\qquad\qquad \cdot \left( \frac{\No(\n_0)^{\frac{1}{4}}}{\No(\mathfrak{m}_1(g))^{\frac{1}{4}}}+\abs{R}_{\infty}^{\frac{1}{2}}\No(\n_0\mathfrak{m}_1(g))^{\frac{1}{4}} \right) \nonumber \\
	&\ll_F& (\abs{\frac{T}{y}}_{\infty}\No(\n))^{\epsilon} \bigg(\abs{R}_{\infty}^{-\frac{1}{4}} \frac{\abs{T}_{\infty}^{\frac{1}{6}}\No(\n_0)^{\frac{1}{4}}}{\No(\mathfrak{m}_1(g))^{\frac{1}{4}}}+\No(\n_0)^{\frac{1}{2}}\nonumber \\
	&&\qquad\qquad+\abs{T}_{\infty}^{\frac{1}{6}}\abs{R}_{\infty}^{\frac{1}{4}}\No(\n_0\mathfrak{m}_1(g))^{\frac{1}{4}}+\abs{R}_{\infty}^{\frac{1}{2}}\No(\n_0\mathfrak{m}_1(g))^{\frac{1}{2}} \bigg). \nonumber
\end{eqnarray}
Using $\abs{R}_{\infty}< \No(\imath)^{-1} \ll \No(\n_0\mathfrak{m}_1(g))^{-1}$ to remove $\abs{R}_{\infty}^{-\frac{1}{4}}$ concludes the proof. It is obvious from Lemma~\ref{lm:est_error} that the error gets absorbed.
\end{proof}

\section{Bounds in the bulk} \label{sec:amp_for_cusp}

After having obtained estimates for automorphic forms near the cusps we now have to determine their size in the bulk. To do so we will use a variant of the so called amplification method. More precisely we will define an integral operator which approximates an spectral projector on a certain subspace of $L^2(G(F)\setminus G(\A_F))$ related to the automorphic form under consideration. An geometric estimation of the kernel then yields the desired estimate.

Let $(\pi,V_{\pi})$ be a cuspidal automorphic representation with new vector $v^{\circ}$ and associated newform $\phi_{\circ} = \sigma(v^{\circ})$. Throughout this section we fix an square free ideal $\mathfrak{q}$ such that all the units that are quadratic residues modulo $\mathfrak{q}$ are indeed contained in $(\mathcal{O}_F^{\times})^2$. We will further assume that $(\mathfrak{q},\n)=1$. Later on we will see that one can actually construct such an ideal.

\subsection{Amplification and the spectral expansion} \label{sec:spec_exp_amp}

Let $\phi = \phi_{\circ}^{\mathfrak{L}} = \sigma^{\mathfrak{L}}(v_{\mathfrak{L}}^{\circ})$.  By Corollary~\ref{cor:reduction_to_f_dom} it is enough to consider $\phi(g)$ for 
\begin{equation}
	g= a(\theta_i) g'n(x) a(y), \text{ for } n(x)a(y) \in \mathcal{F}_{\n_2}, g'=kh_{\n} \in \mathcal{J}_{\n}. \nonumber
\end{equation}
Therefore, we further define $\phi'=\phi(\cdot h_{\n})$. This function is $K_1'(\n) = h_{\n} K_1(\n) h_{\n}^{-1}$ invariant and can be considered as an element of the Hilbert space
\begin{equation}
	L^2(X) = L^2(G(F) \setminus G(\A_F) /K'_1(\n), \omega_{\pi}) \subset L^2(G(F)\setminus G(\A_F)). \nonumber
\end{equation}
Furthermore, we put $w^{\circ} = \pi^{\mathfrak{L}}(h_{\n})v^{\circ}_{\mathfrak{L}}$. Then $\phi' = \sigma^{\mathfrak{L}}(w^{\circ})$. We will bound $\psi'$ on elements $g=a(\theta_i) g' n(x) a(y)$  with $g'\in K_{\n}h_{\n}^{-1}$ and $n(x)a(y) \in \mathcal{F}_{\n_2}$.

We will now define the kernel function which serves as an approximate spectral projector mentioned in the introduction to this section. We do this place by place and immediately give some basic properties.

Let $\nu$ be an archimedean place. Then we define 
\begin{equation}
	f_{\nu}(g_{\nu}) = k_{\nu}\big(u_{\nu}(g_{\nu}.i_{\nu},i_{\nu})\big), \nonumber
\end{equation}
for $k_{\nu}$ as in \cite[Lemma~9]{BHMM16}. By uniqueness of the spherical vector we have
\begin{equation}
	R(f_{\nu}) w_{\nu}^{\circ} = c_{\nu}(\pi_{\nu}) w_{\nu}^{\circ}. \nonumber
\end{equation}
The number $c_{\nu}(\pi_{\nu})$ depends only on the equivalence class of $\pi_{\nu}$ and is given by the spherical transform of $f_{\nu}$ at $\pi_{\nu}$. By a suitable parametrization of spherical representations of $G(F_{\nu})$ one relates this to the classical Selberg/Harish-Chandra transform of $k_{\nu}$. Therefore, we have 
\begin{equation}
	c_{\nu}(\pi_{\nu}) \gg 1 \label{eq:lower_bound_nu_for_c}
\end{equation}
by \cite[Lemma~9]{BHMM16}.

For $\p\vert \n$ we define
\begin{equation}
	f_{\p}(g_{\p}) = \abs{\det(g_{\p})}^{ia_{\p}}\Phi_{\pi_{\p}'}'(g_{\p}). \nonumber
\end{equation} 
Where $\Phi_{\pi_{\p}'}'$ is the chopped of matrix coefficient as defined in \cite[Section~2.4]{Sa15}. By construction (see \cite[Proposition~2.13]{Sa15}) there is $\delta_{\pi_{\p}'} \gg q_{\p}^{-n_{1,\p}-m_{1,\p}}$ such that
\begin{equation}
	R_{\p}(f_{\p})w_{\p}^{\circ} = \int_{Z(F_{\p})\setminus G(F_{\p})} f_{\p}(g) \pi_{\p}(g)w_{\p}^{\circ} d\mu_{\p}(g) = \delta_{\pi_{\p}'}w_{\p}^{\circ}. \nonumber
\end{equation}
Let us remark that 
\begin{eqnarray}
	\abs{f_{\p}(g)} &\leq& 1 \text{ for all } g\in G(F_{\p}), \nonumber \\
	\text{supp}(f_{\p}) &=& \begin{cases} Z(F_{\p})K_{\p} &\text{ if $n_{\p}$ is even,} \\ Z(F_{\p})K_{\p}^0(1) &\text{ else.} \end{cases} \nonumber
\end{eqnarray}

For $\p\vert \mathfrak{q}$ define
\begin{equation}
	\tilde{K}_{0,\p}(1) = \left\{ \left( \begin{matrix} a&b \\ c & d \end{matrix} \right)\in K_{0,\p}(1) \colon a-d\in \varpi_{\p} \op_{\p} \right\}. \nonumber
\end{equation}
Then we put 
\begin{equation}
	f_{\p}(g_{\p}) = \begin{cases} \text{vol}(Z(\op_{\p})\setminus \tilde{K}_{0,\p}(1))^{-1} \omega_{\pi_{\p}}^{-1}(z) &\text{ if } g_{\p} = zk\in Z(F_{\p})\tilde{K}_{0,\p}(1), \\ 0 &\text{ else}. \end{cases} \nonumber
\end{equation}
Now, since $w^{\circ}_{\p}$ is $K_{\p}$-fixed we see:
\begin{eqnarray}
	R_{\p}(f_{\p}) w^{\circ}_{\p} &=& \int_{Z(F_{\p})\setminus G(F_{\p})} f_{\p}\pi_{\p}(g)w_{\p}^{\circ}d\mu_{\p}(g) \nonumber\\ 
	&=& \text{vol}(Z(\op_{\p})\setminus \tilde{K}_{0,\p}(1))^{-1} \int_{Z(\op_{\p})\setminus \tilde{K}_{0,\p}(1)} \omega_{\pi_{\p}}(z)^{-1} \pi_{\p}(zk)w_{\p}^{\circ} d\mu_{\p}(zk) = w_{\p}^{\circ}. \nonumber
\end{eqnarray}
We also have the estimate
\begin{equation}
	\abs{f_{\p}} \leq [ K_{\p}:\tilde{K}_{0,\p}(1)] \ll q_{\p}^{2+\epsilon}.\nonumber
\end{equation}

We will treat the remaining places at once. To do so we set $S_{\text{ur}} = \{ \p\colon (\p,\mathfrak{q}\n) =1 \}$. And we define the unramified Hecke algebra
\begin{equation}
	\mathcal{H}_{\text{ur}} = \langle \{ \kappa_{\text{ur}}=\otimes_{\p\in S_{\text{ur}}} \kappa_{\p} \colon  \kappa_{\p} \in \mathcal{C}_{c}^{\infty}(G(F_{\p}), \omega_{\pi_{\p}}) \text{ such that } \kappa_{\p}(K_{\p} gK_{\p}) = \kappa_{\p}(g) \} \rangle_{\C}. \nonumber
\end{equation}
This is a commutative algebra by \cite[Theorem~4.6.1]{Bu96}. To an integral ideal $\mathfrak{c}$ we associate the special element
\begin{equation}
	\kappa_{\mathfrak{c}} = \otimes_{\p \in S_{\text{ur}}} \kappa_{\p,v_{\p}(\mathfrak{c})} \in \mathcal{H}_{\text{ur}} \nonumber
\end{equation}
 where
\begin{equation}
	\kappa_{\p,k}(g) = \begin{cases} \omega_{\pi_{\p}}(z)^{-1} & \text{ for } g=z\ \in Z(F_{\p})X_{\p,k}, \\ 0 &\text{ else.} \end{cases} \nonumber
\end{equation}
This is well defined since the central character is unramified at the places under consideration. This function is constructed such that $\pi(\mathbbm{1}_{X_{\p,k}}) = R(\kappa_{\p,k})$. Therefore, we have for $w^{\circ}_{\text{ur}} = \otimes_{\p\in S_{\text{ur}}} w_{\p}^{\circ}$ that
\begin{equation}
	R(\kappa_{\mathfrak{c}}) w_{\text{ur}}^{\circ} = \lambda(\mathfrak{c})  w_{\text{ur}}^{\circ}. \nonumber
\end{equation}

Now let us fix a large parameter $L$ such that $\No(q) \ll (\log L)^A$ for some constant $A$. We then define the sets 
\begin{eqnarray}
	\mathcal{P}_{\mathfrak{q}} &=& \{\mathfrak{a} \colon \mathfrak{a} = ( \alpha ) \text{ for } \alpha \in F_+^{\times}\cap (1+\mathfrak{q}) \}, \nonumber\\
	\mathcal{J}(\mathfrak{q}) &=& \{ \mathfrak{a} \colon (\mathfrak{a},\mathfrak{q})=1\} \text{ and } \nonumber\\
	\mathcal{P}(L) &=& \{ \alpha \in \mathcal{O}_F \colon (\alpha) \in \mathcal{P}_{\mathfrak{q}} \text{ is a prime ideal with } \No(\alpha) \in [L,2L]\text{ and } ((\alpha),\n)=1\}/\sim. \nonumber
\end{eqnarray}
In the last definition we say  $\alpha \sim\beta$ if $(\alpha) =(\beta)$. We identify $\mathcal{P}(L)$ with a suitable fundamental domain for $\sim$. We can arrange that $\alpha_v \asymp L^{[F:\Q]}$ for all $\nu$ and all $\alpha\in\mathcal{P}(L)$.

We now need a lower bound for $\sharp \mathcal{P}(L)$. Since we can not assume $\mathfrak{q}$ to be fixed (it might depend on $\n$) we need a stronger argument than the one given in \cite{BHMM16}. The following variation of the generalized Siegel--Walfisz theorem does the job.

\begin{lemma} \label{lm:lower_length_amp}
If $\No(q) \ll (\log L)^A$ we have
\begin{equation}
	\frac{L}{\No(\mathfrak{q}) \log(L)} \ll_{F,A} \sharp \mathcal{P}(L) \ll_{F,A} \frac{L}{\log(L)}. \label{eq:lower_length_amp}
\end{equation}
\end{lemma}
This is a very lazy estimate but it uses some heavy machinery. So we will sketch the proof.
\begin{proof}
Let $Cl_F^{\mathfrak{q}} = \mathcal{J}(\mathfrak{q})/\mathcal{P}_{\mathfrak{q}}$ be the ray class group.  The explicit formula \cite[VI, Theorem~1]{La13} for the cardinality of $Cl_F^{\mathfrak{q}}$ implies
\begin{equation}
	\sharp Cl_F^{\mathfrak{q}} \ll \No(\mathfrak{q}). \nonumber
\end{equation}
For our purposes this is enough. 

The statement then follows immediately from \cite[Korollar~1.3]{Hi80}.
\end{proof}

\begin{rem}
We could also work with the weaker assumption $\No(q) \ll_{\epsilon} \No(\n)^{\epsilon}$. In this case we can still obtain a good lower bound for $\sharp \mathcal{P}(L)$ using a version of Linnik's theorem over number fields.
\end{rem}

To $\alpha \in \mathcal{O}_{F}$ we associate the numbers
\begin{equation}
	x_{\alpha} = \frac{\overline{\lambda((\alpha))}}{\abs{\lambda((\alpha))}}. \nonumber
\end{equation}
We finally define the unramified test function to be
\begin{equation}
	f_{\text{ur}} = \left( \sum_{\alpha\in \mathcal{P}(L)} \frac{x_{\alpha}\kappa_{\alpha}}{\sqrt{\mathcal{N}(\alpha)}} \right)\left( \sum_{\alpha\in \mathcal{P}(L)} \frac{x_{\alpha}\kappa_{\alpha}}{\sqrt{\mathcal{N}(\alpha)}} \right)^* + \left( \sum_{\alpha\in \mathcal{P}(L)} \frac{x_{\alpha^2}\kappa_{\alpha^2}}{\sqrt{\mathcal{N}(\alpha^2)}} \right) \left( \sum_{\alpha\in \mathcal{P}(L)} \frac{x_{\alpha^2}\kappa_{\alpha^2}}{\sqrt{\mathcal{N}(\alpha^2)}} \right)^*. \nonumber
\end{equation}
This defines an operator $R(f_{\text{ur}})$ such that
\begin{eqnarray}
	R(f_{\text{ur}})w_{\text{ur}}^{\circ} = \underbrace{\left[\left( \sum_{\alpha\in \mathcal{P}(L)} \frac{\abs{\lambda((\alpha))}}{\sqrt{\mathcal{N}(\alpha)}} \right)^2 +\left( \sum_{\alpha\in \mathcal{P}(L)} \frac{\abs{\lambda((\alpha^2))}}{\sqrt{\mathcal{N}(\alpha^2)}} \right)^2 \right]}_{=c_{\text{ur}}} w_{\text{ur}}^{\circ}
	= c_{\text{ur}}w_{\text{ur}}^{\circ}. \nonumber
\end{eqnarray}
Using \cite[Proposition~4.6.4/4.6.6]{Bu96} and \eqref{eq:lower_length_amp} and arguing as in \cite[(9.17)]{BHMM16} one gets
\begin{equation}
	c_{\text{ur}} \gg \frac{L^2}{\No(\mathfrak{q})^2\log(L)^2}. \label{eq:ur_low_bound_c}
\end{equation}
On the other hand we can linearise $f_{\text{ur}}$ as usual to obtain
\begin{equation}
	f_{\text{ur}} = \sum_{\alpha\in \mathcal{O}_F} y_\alpha \frac{\kappa_{\alpha}}{\sqrt{\mathcal{N}(\alpha)}}. \label{eq:linearization_of_fur}
\end{equation}
The coefficients $y_{\alpha}$ are very similar in spirit to the coefficients $w_m$ in \cite[(9.16)]{BHMM16}. Indeed
\begin{equation}
	y_{\alpha} = \begin{cases} \sum_{\alpha' \in \mathcal{P}(L)} \abs{x_{\alpha'}}^2 \omega^{-1}_{\pi_{(\alpha')}}(\varpi_{(\alpha')})+ \abs{x_{\alpha'^2}}^2 \omega^{-1}_{\pi_{(\alpha')}}(\varpi_{(\alpha')}^2) &\text{ if }\alpha=1 \\ x_{\alpha_1}\overline{x_{\alpha_2}}+\delta_{\alpha_1=\alpha_2} \omega_{\pi_{(\alpha_1)}}^{-1}(\varpi_{(\alpha_1)})x_{\alpha_1^2}\overline{x_{\alpha_2^2}} &\text{ if } \alpha=\alpha_1\alpha_2 \text{ for } \alpha_1,\alpha_2\in \mathcal{P}(L), \\  x_{\alpha_1^2}\overline{x_{\alpha_2^2}}  &\text{ if } \alpha=\alpha_1^2\alpha_2^2 \text{ for } \alpha_1,\alpha_2\in \mathcal{P}(L)\\ 0 &\text{ else}.\end{cases} \nonumber
\end{equation}
Thus, most importantly we have 
\begin{equation}
	\abs{y_\alpha} \ll \begin{cases} L &\text{ if } \alpha =1, \\  1 & \text{ if } \alpha = \alpha_1^j \alpha_2^j \text{ for some } j=1,2 \text{ and } \alpha_1,\alpha_2\in \mathcal{P}(L),\\ 0 & \text{ else} .\end{cases} \nonumber
\end{equation}
One compares this to \cite[p. 29]{Sa15_2} and \cite[p. 27]{BHMM16} and notes the similarity.

Combining everything we define 
\begin{equation}
	f=\otimes_{\nu} f_{\nu} \otimes_{\p\vert\mathfrak{q}\n} f_{\p} \otimes f_{\text{ur}}. \nonumber
\end{equation}
Associated to this function there is the integral operator
\begin{eqnarray}
	R(f)\colon L^2(G(F)\setminus G(\A_F),\omega_{\pi}) &\to& L^2(G(F)\setminus G(\A_F),\omega_{\pi}) \nonumber \\
	\phi &\mapsto& \left[x\mapsto  \int_{Z(\A_F)\setminus G(\A_F)} f(g)\phi(gx) dg \right]. \nonumber
\end{eqnarray}
In particular we have
\begin{equation}
	R(f)\phi' = \sigma^{\mathfrak{L}}\left( \int_{Z(\A_F)\setminus G(\A_F)} f(g)\pi^{\mathfrak{L}}(g)w^{\circ}dg \right) = c_{\text{ur}} \prod_{\nu} c_{\nu}(\pi_{\nu}) \prod_{\p\vert\n} \delta_{\pi_{\p}'} \phi'. \nonumber
\end{equation}

The corresponding automorphic kernel is given by
\begin{equation}
	K_f(g_1,g_2) = \sum_{\gamma\in Z(F)\setminus G(F)} f(g_1^{-1}\gamma g_2). \nonumber
\end{equation}

The spectral expansion of $K_f$ will enable us to bound the sup-norm of $\phi'$ in terms of the geometric definition of $K_f$. Let us work out the spectral expansion in detail.

We decompose
\begin{equation}
	K_f = K_{\text{cusp}}+K_{\text{sp}}+K_{\text{cont}}.
\end{equation}

Let us first deal with the cuspidal part. 
\begin{lemma}
For any $g\in G(\A_F)$ we have
\begin{equation}
	0\leq \frac{L^{2-\epsilon}}{\No(\mathfrak{q})^2\No(\n_1)\No(\mathfrak{m}_1)}\abs{\phi'(g)}^2 \ll K_{\text{cusp}}(g,g). \nonumber 
\end{equation}
\end{lemma}
\begin{proof}
We begin by fixing a basis $\mathcal{B}_{\text{cusp}}$ for $L^2_0(X)$ containing $\phi'$ and consisting of $R(F)$ eigenfunctions. This is possible by a standard multiplicity one argument. For $\Psi\in \mathcal{B}_{\text{cusp}}$ let $c_{\Psi}$ be the associated $R(f)$-eigenvalue. Then we obtain
\begin{equation}
	K_{\text{cusp}}(h,g) = \sum_{\Psi\in \mathcal{B}_{\text{cusp}}} \SP{K_{\text{cusp}}(\cdot,g)}{\Psi}_{L^2(X)} \Psi(h) = \sum_{\Psi\in \mathcal{B}_{\text{cusp}}} \overline{c_{\Psi}\Psi(g)}\Psi(h). \nonumber
\end{equation}
We can choose $\mathcal{B}_{\text{cusp}}$ in such a way that for each $\Psi$ there is a cuspidal automorphic representation $(\pi_{\Psi}, V_{\Psi})$ and $\Psi = \sigma_{\Psi}(v)$ for some pure tensor $v\in V_{\Psi}$. Then we have
\begin{equation}
	c_{\Psi} = \delta_{\Psi} c_{\Psi,\text{ur}}\prod_{\nu} c_{\nu}(\pi_{\Psi,\nu}). \nonumber
\end{equation}
By \cite[Corollary 2.16]{Sa15} we have $\delta_{\Psi} \in \{0,\delta_{\pi}\}$. In particular $\delta_{\Psi}\geq 0$. At the archimedean places positivity of $c_{\nu}(\pi_{\Psi,\nu})$ is ensured by the definition of $k_{\nu}$. Finally also $c_{\text{ur}}$ must obviously be positive since $R(f_{\text{ur}})$ is a positive operator. Therefore
$c_{\Psi}\geq 0$ for all $\Psi\in \mathcal{B}_{\text{cusp}}$. 
An explicit lower bound for $c_{\phi'}$ follows from \eqref{eq:ur_low_bound_c}, \eqref{eq:lower_bound_nu_for_c}, and \cite[Proposition~2.13]{Sa15}. We then conclude by dropping all unnecessary terms. 
\end{proof}

The argument for the continuous part is quite similar. We obtain
\begin{lemma}
For $g\in G(\A_F)$ one has
\begin{equation}
	K_{\text{cont}} (g,g) \geq 0. \nonumber
\end{equation}
\end{lemma}
\begin{proof}
Using the theory of Eisenstein series we have the expansion
\begin{equation}
	K_{cont}(h,g) = \frac{1}{4\pi} \sum_{\Psi_1,\Psi_2 \in \mathcal{B}_{\tilde{\mathbf{H}}}} \int_{-\infty}^{\infty} \SP{R(f) \Psi_2(iy)}{\Psi_1(iy)}_{\tilde{\mathbf{H}}(iy)} E_{\Psi_1}(iy,h) \overline{E_{\Psi_2}(iy,g)} dy.\label{eq:cont_kernel_exp}
\end{equation}
This is \cite[(5.21)]{GJ79}. Let us briefly recall the notation. We define the space
\begin{eqnarray}
	\tilde{\mathbf{H}}(s) = \bigg\{\Psi\colon G(\A_F) \to \C &\colon& \Psi\left[\left(\begin{matrix} \alpha au &x \\ 0 & \beta a v \end{matrix} \right) g\right] = \omega_{\pi}(a)\abs{\frac{u}{v}}_{\infty}^{s+\frac{1}{2}}\Psi(g) \nonumber\\
	&&\text{for }\alpha,\beta\in F^{\times}, a\in \A_F^{\times}, u,v\in F_{\infty}^+,  \nonumber\\
	&&\qquad\qquad \int_K \int_{F^{\times}\setminus F^{0}(\A_F)} \abs{\Psi(a(y)k)}^2 dy d\mu_K(k)   \bigg\} \nonumber
\end{eqnarray}
We then have a representation $(\pi_s,\tilde{\mathbf{H}}(s))$ of $G(\A_F)$ where $G(\A_F)$ acts by right translation. For $s\in i\R$ we have the inner product
\begin{equation}
	\SP{\Psi_1}{\Psi_2}_{\tilde{\mathbf{H}}(s)} \int_{\A_F^{\times}} \int_K \Psi_1( a(y)k)\overline{\Psi_2(a(y)k)} d\mu_{\A_F^{\times}}^{\times}(y) d\mu_K(k). \nonumber 
\end{equation}
We can also view $\tilde{\mathbf{H}}(s)$ as a trivial holomorphic fibre bundle over $\tilde{\mathbf{H}} = \tilde{\mathbf{H}}(0)$. For $\phi\in \tilde{\mathbf{H}}$ we then define $\Psi(s) = \Psi\cdot H(\cdot)^{s} \in \tilde{\mathbf{H}}(s)$. Where, 
\begin{equation}
		H\left( \left( \begin{matrix} 1&x \\ 0&1 \end{matrix} \right) \left( \begin{matrix} a&0 \\ 0&b \end{matrix} \right) k \right) = \abs{\frac{a}{b}}_{\A_F} \text{ for all }k\in K, \nonumber
\end{equation}
is naturally defined via the Iwasawa decomposition of $G(\A_F)$. Further, to $\Psi\in \tilde{\mathbf{H}}$ we associate the Eisenstein series
\begin{equation}
	E_{\Psi}(s,g) = \sum_{\gamma \in B(F)\setminus G(F)}\Psi(\gamma g)H(\gamma g)^s. \nonumber
\end{equation}
The sum in \eqref{eq:cont_kernel_exp} is taken over an orthonormal basis $\mathcal{B}_{\tilde{\mathbf{H}}}$ for $\tilde{\mathbf{H}}$.

As earlier it is no problem to choose this basis to consist of $R(F)$ eigenfunctions. For $\Psi\in \mathcal{B}_{\tilde{\mathbf{H}}}$ we denote the corresponding $R(f)$-eigenvalue by $c_{\Psi}(0)$. Note that then also $\Psi(s)$ is an $R(f)$-eigenfunction but the eigenvalue may depend on $s$. Thus by putting $h=g$ we obtain
\begin{equation}
	K_{\text{cont}}(g,g)  = \frac{1}{4\pi} \sum_{\Psi\in\mathcal{B}_{\tilde{\mathbf{H}}}}\int_{-\infty}^{\infty}  c_{\Psi}(iy) \abs{E_{\Psi}(iy,g)}^2 dy. \nonumber
\end{equation}
We can now argue as before using the construction of $f$ to show that $c_{\Psi}(s)\geq 0$ for all $\Psi$. This concludes the proof.
\end{proof}

Finally we also treat the residual part of the spectrum. 
\begin{lemma}
As long as $\mathfrak{n}_0\neq \mathcal{O}_{F}$ we have
\begin{equation}
	K_{\text{sp}}(h,g) = 0 \nonumber
\end{equation}
for any $g,h\in G(\A_F)$. Otherwise we still have $K_{\text{sp}}(g,g) \geq 0$ and the only contribution comes from characters $\chi^2 = \omega_{\pi}$ with $a(\chi_{\p}) \leq 1$ at $\p\vert\mathfrak{q}\n_1[\n_1]_{\n_0}^{-1}$ and $a(\chi_{\p})=0$ otherwise.
\end{lemma}
\begin{proof}
We start from the spectral expansion of $K_{\text{sp}}$. This reads
\begin{equation}
	K_{\text{sp}}(h,g) = \frac{1}{\vol(Z(\A_F)G(F)\setminus G(\A_F))} \sum_{\chi^2 = \omega_{\pi}} \chi(\det(h)) \overline{\chi(\det(g))} \int_{Z(\A_F) \setminus G(\A_F)} f(x) \chi(\det(x)))dx. \nonumber
\end{equation}
Since the character $\chi$ factors and also $f$ is almost a pure tensor the last integral factors in the local integrals
\begin{equation}
	I_{\p}(\chi_{\p}) = \int_{Z(F_{\p})\setminus G(F_{\p})} f_{\p}(g)\chi_{\p}(\det(g))dg \text{ if } \p\vert \n\mathfrak{q} \nonumber
\end{equation}
and the unramified part $I_{\text{ur}}(\chi_{\text{ur}})$. By Lemma~\ref{lm:unram_resid_integral} it is clear that $I_{\text{ur}}(\chi_{\text{ur}})\geq 0$. The lemma then follows from the evaluation of the itnegrals $I_{\p}(\chi_{\p})$ given in Lemma~\ref{lm:q_resid_integral}~and~\ref{lm:ram_resid_integral}.
\end{proof}

By combining the last three lemmata with the definition of $K_f$ we conclude
\begin{equation}
	\abs{\phi'}^2 \ll_{\epsilon} L^{-2+\epsilon} \No(\mathfrak{q})^2\No(\n_1\mathfrak{m}_1)\sum_{\gamma \in Z(F)\setminus G(F)} \abs{f(g^{-1}\gamma g)}. \label{eq:combination_of_kernel_drop}
\end{equation}
This gives an upper bound for $\phi'$ in terms of the geometry of $G(F)$ and the test function $f$. We will estimate this further in the next section.

\subsection{Estimating the geometric expansion}

In this subsection we prove an upper bound for $\phi'$ which is good in the bulk. This will be done by estimating the right hand side of \eqref{eq:combination_of_kernel_drop}.

\begin{prop} \label{prop:bulk_arb_nf}
Assume we have $(\mathfrak{q},\n)=1$ and $\No(\mathfrak{q}) \ll \log(\No(\n))^A$. Then for 
\begin{equation}
	g=a(\theta_i)g'n(x)a(y) \text{ with } g'\in K_{\n}h_{\n}^{-1} \text{ and } n(x)a(y)\in \mathcal{F}_{\n_2}\nonumber
\end{equation}
we have
\begin{eqnarray}
	\abs{\phi'(g)}^2 \ll (\abs{T}_{\infty} \mathcal{N}(\n))^{\epsilon} \No(\mathfrak{q})^{4+\epsilon} &\bigg(& \abs{T}_{\infty}^{\frac{5}{6}}\mathcal{N}(\n_1)^{\frac{2}{3}}\mathcal{N}(\n_0)^{\frac{1}{3}}\mathcal{N}(\mathfrak{m}_1)\nonumber\\
	&& + \abs{T}_{\R}^{\frac{1}{2}}\abs{T}_{\C}\mathcal{N}(\n_1)^{\frac{1}{2}}\mathcal{N}(\n_0)^{\frac{1}{2}}\mathcal{N}(\mathfrak{m}_1) \nonumber\\
	&& + \abs{T}_{\infty}^{\frac{1}{2}}\mathcal{N}(\n_1)\mathcal{N}(\mathfrak{m}_1)\abs{y}_{\infty} \bigg). \label{eq:bulk_est_for_general} 
\end{eqnarray}
\end{prop}
The only thing we will have to do is to exploit the support properties of $f$ to reduce the estimate to the counting problem solved in \cite{BHMM16}. Comparing this result to \cite[Theorem~1]{BHMM16} and \cite[Theorem~3.2]{Sa15_2} shows that the exponents here are indeed as expected.
\begin{proof}
To save chalk we put
\begin{equation}
	k(u(\gamma P,P)) = \prod_v k_v(u_v(\gamma_v P_v, P_v)) \text{ with } P_v = n(x_v) a(y_v). i_v. \nonumber
\end{equation}
By inserting the linearization of $f_{ur}$ given in \eqref{eq:linearization_of_fur} into \eqref{eq:combination_of_kernel_drop} yields 
\begin{eqnarray}
	\abs{\phi'(g)}^2 &\ll&  L^{-2+\epsilon} \No(\n_1\mathfrak{m}_1)\sum_{0\neq \alpha\in \mathcal{O}_F} \frac{\abs{y_{\alpha}}}{\sqrt{\mathcal{N}(\alpha)}} \nonumber \\
	&&\qquad\qquad\qquad\qquad\cdot\sum_{\gamma\in Z(F) \setminus G(F)} \abs{ \kappa_{(\alpha)}\prod_{\p\vert\mathfrak{q}\n}f_{\p}}(g'^{-1}a(\theta_i^{-1})\gamma a(\theta_i)g')\abs{k(u(\gamma P,P))} \nonumber
\end{eqnarray}
Let us analyze the support of $f_{\p}$ and $\kappa_{(\alpha)}$ place by place. At this point we will also have to exploit the special structure of $g$.

First note that if $\p \nmid \n$ we have $g'_{\p}=1$. This case consists of two sub cases Namely 
\begin{equation}
	a(\theta_i^{-1})\gamma a(\theta_i) \in \begin{cases} Z(F_{\p}) \tilde{K}_{0,\p}(1) &\text{ if } \p\mid \mathfrak{q}, \\ Z(F_{\p})K_{\p} a(\varpi_{\p}^{v_{\p}(\alpha)}) K_{\p} & \text{ else.}  \end{cases} \nonumber
\end{equation}

If $\p \mid \n$ then we use Lemma~\ref{lm:claim3} to see that $g'_{\p} \in \omega K_{\p}^0(1)$ if $p\mid \n_2$ and $g'_{\p} K_{\p}$ otherwise. Using the support property of $f_{\p}$ we conclude that
\begin{equation}
	a(\theta_i^{-1})\gamma a(\theta_i) \in \begin{cases} Z(F_{\p}) K_{\p} &\text{ if } \p \nmid \n_2, \\ Z(F_{\p}) \underbrace{\omega K_{\p}^0(1) \omega^{-1}}_{=K_{0,\p}(1)} &\text{ if } \p\mid \n_2. \end{cases} \nonumber
\end{equation}

It is now straight forward to choose a suitable representative for $\gamma\in Z(F)\setminus GL_2(F)$ such that we arrive at the analogue of \cite[(9.20)]{BHMM16}. In our case this reads
\begin{equation}
	\abs{\phi'(g)}^2 \ll_{\epsilon} \No(\mathfrak{q})^{2+\epsilon} L^{-2+\epsilon}\mathcal{N}(\n_1 \mathfrak{m}_1) \sum_{0\neq \alpha\in \mathcal{O}_F} \frac{\abs{y_{\alpha}}}{\sqrt{\mathcal{N}(\alpha)}} \sum_{\gamma\in \Gamma(i,\alpha)} \abs{k(u(\gamma P,P))} \nonumber
\end{equation}
with
\begin{equation}
	\Gamma(i,\alpha) = \left\{ \left(\begin{matrix} a&b \\ c& d \end{matrix}\right)\in GL_2(F) \colon a,d\in \mathcal{O}_F, a-d\in \mathfrak{q}, b\in \theta_i\mathcal{O}_F, c\in \theta_i^{-1}\n\mathfrak{q}, ad-bc=\alpha  \right\}. \nonumber
\end{equation}

Since our coefficients $y_{\alpha}$ have the same properties as the corresponding $w_m$ in \cite{BHMM16} we can replicate the argument from \cite[p. 26]{BHMM16}. One quickly sees, that this argument does not produce any new $\mathfrak{q}$ dependence. We arrive at
\begin{equation}
	\abs{\phi'(g)}^2 \ll_{\epsilon} \No(\mathfrak{q})^{2+\epsilon} L^{\epsilon}\mathcal{N}(\n_1 \mathfrak{m}_1) \sum_{\substack{\underline{k}\in \Z^{\sharp\{v\}}, \\ T_v^{-2}\leq \delta_v = 2^{k_v}\leq 4}} \frac{\abs{T}^{\frac{1}{2}}_{\infty}}{\abs{\delta}_{\infty}^{\frac{1}{4}}}\left( \frac{M(L,0,\delta)}{L}+ \frac{M(L,1,\delta)}{L^3}+ \frac{M(L,2,\delta)}{L^4} \right) \label{eq:where_to_insert_M}
\end{equation}
for
\begin{equation}
	M(L,j,\delta) = \sum_{\alpha_1,\alpha_2\in \mathcal{P}(L)} \sharp\left\{\gamma \in \Gamma(i,\alpha_1^j\alpha_2^j) \colon u_v(\gamma_vP_v,P_v)\leq \delta_v \text{ for all } v  \right\}. \nonumber
\end{equation}
This is analogous to  \cite[(9.24)]{BHMM16}.

The last sum contains $\ll_{\epsilon} \abs{T}_{\infty}^{\epsilon}$ terms, so that we can estimate it trivially. Further we note that $\abs{T_v}_{v}^{-2} \leq \abs{\delta_v}_{v}\ll 1$. This allows us to use the bounds for $M(L,j,\delta)$ as summarized on \cite[p. 37]{BHMM16} to estimate\footnote{As mentioned in \cite{BHMM16} the results counting the elements in $M(L,j,\delta)$ totally ignore the conditions depending on $\mathfrak{q}$. Therefore, using these bounds does not generate any new $\mathfrak{q}$ dependence.}
\begin{equation}
	\abs{\phi'(g)}^2 \ll (L\abs{T}_{\infty} \mathcal{N}(\n))^{\epsilon}\No(\mathfrak{q})^{2+\epsilon} \mathcal{N}(\n_1 \mathfrak{m}_1) \left( \frac{\abs{T}_{\infty}}{L}+\abs{T}_{\infty}^{\frac{1}{2}}\abs{y}_{\infty}+\frac{L^2\abs{T}_{\infty}}{\mathcal{N}\n_2} + \frac{\abs{T}^{\frac{1}{2}}_{\R}\abs{T}_{\C}}{(\mathcal{N}\n_2)^{\frac{1}{2}}} \right). \nonumber
\end{equation}
Choosing $L=\abs{T}_{\infty}^{\frac{1}{6}}(\mathcal{N}\n_2)^{\frac{1}{3}}\No(\n)^{\epsilon}$ and noting that $\n_2 \n_0 = \n_1$ leads to \eqref{eq:bulk_est_for_general}. Note that we include factor $\No(\n)^{\epsilon}$ in the definition of $L$ to make sure that $\No(\mathfrak{q}) \ll_{\epsilon} \log(L)^A$. 
\end{proof}

As in \cite{BHMM16} we can give another estimate for non-totally-real number fields. 

\begin{prop} \label{prop:bulk_bound_cmpF}
Let $C\leq \No(\mathfrak{q})\ll \log(\No(\n))^A$, where $C$ is an explicitly computable constant depending only on the field $F$. Further let $F^{\R}$ be the maximal totally real subfield of $F$ and let $m = [F:F^{\R}] \geq 2$. Then we have
\begin{equation}
	\abs{\phi'(g)}^2 \ll (\abs{T}_{\infty} \mathcal{N}(\n))^{\epsilon}\No(\mathfrak{q})^{4+\epsilon} \mathcal{N}(\n_1 \mathfrak{m}_1)\abs{T}_{\infty}\left( \abs{T}_{\infty}^{\frac{-1}{4m-4}}+(\abs{T}_{\infty} \No\n_2)^{\frac{-1}{4m-2}}+\frac{\abs{y}_{\infty}}{\abs{T}_{\infty}^{\frac{1}{2}}}  \right). \nonumber 
\end{equation}
\end{prop}
\begin{proof}
To see this one uses the second list in \cite[p.37]{BHMM16} together with \eqref{eq:where_to_insert_M}. This yields
\begin{equation}
	\abs{\phi'(g)}^2 \ll (\abs{T}_{\infty}L \mathcal{N}(\n))^{\epsilon}\No(\mathfrak{q})^{4+\epsilon} \mathcal{N}(\n_1 \mathfrak{m}_1) \left( \frac{\abs{T}_{\infty}}{L}+\abs{T}_{\infty}^{\frac{1}{2}}\left( \abs{y}_{\infty}+L^{2m-3}+\frac{L^{2m-2}}{\No(\n_2)^{\frac{1}{2}}} \right) \right). \nonumber
\end{equation}
Using $L=\No(\n)^{\epsilon}\min({2^n C_0\abs{T}_{\infty}^\frac{1}{4m-4}},(\No(\n_2)\abs{T}_{\infty})^{\frac{1}{4m-2}})$ completes the proof.\footnote{The Constant $2^nC_0$ makes sure that $L$ is not to small.}
\end{proof}

\section{The endgame} \label{sec:endgame}

In this section we put all the pieces together to prove the theorems stated in the beginning.

\subsection{Constructing the ideal $\mathfrak{q}$}

The section on amplification is dependent on the existence of a square-free ideal $\mathfrak{q}$ which eliminates certain technicalities coming from the unit group of $F$. Here we will show that one can actually construct $\mathfrak{q}$ with the desired properties.

\begin{lemma} \label{lm:choice_of_q_possible}
There is an absolute constant $A>0$ depending only on $F$ such that for any $\n$ there is an ideal $\mathfrak{q}$ satisfying the following two properties. 
\begin{itemize}
	\item We have $C\leq \No(\mathfrak{q})\ll \log(\No(\n))^A$, where $C$ is the absolute constant from Proposition~\ref{prop:bulk_bound_cmpF}.
	\item If $x$ is a quadratic residue modulo $\mathfrak{q}$ then $x\in (\mathcal{O}_F^{\times})^2$.
\end{itemize}
\end{lemma}
\begin{proof}
For $u\in \mathcal{O}_F^{\times} / (\mathcal{O}_F^{\times})^2$ non-trivial, we look at the quadratic extension $F(\sqrt{u}):F$. Then the Galois group is abelian and consists of two elements, say $Gal(F(\sqrt{u})\vert F) = \{ 1,\sigma_u\}$. Since we are dealing with an quadratic extension we know that a prime $\p$ of $F$ is inert in $F(\sqrt{u})$ if and only if the Artin-symbol satisfies
\begin{equation}
	\left(\frac{F(\sqrt{u}):F}{\p}\right) = \sigma_u. \nonumber
\end{equation}
Thus the Chebotarev set 
\begin{equation}
	P_{F(\sqrt{u})\vert F}(\sigma_u) = \left\{ \p \text{ unramified in } F(\sqrt{u}) \colon \left(\frac{F(\sqrt{u}):F}{\p}\right) = \sigma_u \right\} \nonumber
\end{equation}
contains exactly all the primes of $F$ that are inert in $F(\sqrt{u})$. It is a standard fact that $u$ is not a square modulo $\mathfrak{q}_u$ for any $\mathfrak{q}_u\in P_{F(\sqrt{u})\vert F}(\sigma_u)$. Therefore, we want to define
\begin{equation}
	\mathfrak{q} = \prod_{\substack{u\in \mathcal{O}_F^{\times}/(\mathcal{O}_F^{\times})^2,\\ [u] \neq [1]}} \mathfrak{q}_u \nonumber
\end{equation}
for suitably chosen $\mathfrak{q}_u\in  P_{F(\sqrt{u})\vert F}(\sigma_u)$. The rest of the proof is concerned with the problem of choosing $\mathfrak{q}_u$.

To do so we make several definitions. First we define
\begin{equation}
	[\n]_u = \prod_{\substack{ \p\vert \n, \\ \p\in P_{F(\sqrt{u})\vert F}(\sigma_u)}} \p. \nonumber
\end{equation}
Further we number $P_{F(\sqrt{u})\vert F}(\sigma_u) = \{ \p_{u,1}, \p_{u,2}, \cdots \}$ such that $\No(\p_{u,1}) \leq \No(\p_{u,2}) \leq \cdots$.

Now let us consider two cases. First if $\p_{u,1} \nmid [\n]_u$ then we take $\mathfrak{q}_u = \p_{u,1}$. By a version of Linnik's theorem for Chebotarev sets \cite{Za15} we have
\begin{equation}
	\No(\mathfrak{q}_u) \ll_F 1 \nonumber.
\end{equation}

Otherwise, we only consider the worst case 
\begin{equation}
	[\n]_u = \p_{u,1}\cdot ... \cdot \p_{u,k-1}. \nonumber
\end{equation}
Here we define $\mathfrak{q}_u = \p_{u,k}$. It is clear that we only need to show $\No(\mathfrak{q}_u) \ll \log(\No([\n]_u))^A$. But this follows from elementary calculations using Chebotarev's density theorem \cite[Theorem~(13.4), Chapter~VII]{Ne13}.

It is obvious that we can assume $C\leq \No(\mathfrak{q})$.
\end{proof}

\subsection{Proof of the main theorems}

In this subsection let $\phi_{\circ}$ be a $L^2$-normalized Maa\ss\   newform. In other words, it corresponds to a new vector in some cuspidal automorphic representation $(\pi, V_{\pi})$.

\begin{proof}[Proof of Theorem~\ref{th:main_th_1}]
First let us use Corollary~\ref{cor:reduction_to_f_dom}. Thus it is enough to consider $\phi(g) = \phi_{\circ}^{\mathfrak{L}}(g)$ for some $\mathfrak{L}\mid \n$. Further we can fix $1\leq i\leq h_F$ and restrict ourselves to $g=a(\theta_i)g'h_{\n}n(x)a(y)$ with $n(x)a(y) \in \mathcal{F}_{\n_2}$ and $g'h_{\n}\in \mathcal{J}_{\n}$.

If $\abs{y_{\infty}} > \abs{T}^{\frac{1}{3}}\No(\n_2)^{-\frac{1}{3}}$ then Proposition~\ref{pr:whit_exp_est_bad} yields
\begin{equation}
	\abs{\phi(g)}^2 \ll_{F,\epsilon} (\abs{T}_{\infty}\No(\n))^{\epsilon}\bigg( \abs{T}_{\infty}^{\frac{1}{3}}\No(\n_0)+\abs{T}_{\infty}^{\frac{2}{3}}\No(\n_2)^{\frac{1}{3}}\No(\n_0\mathfrak{m}_1)\bigg). \nonumber
\end{equation}

At last we deal with $\abs{y}_{\infty} \leq \abs{T}_{\infty}^{\frac{1}{3}}\No(\n_2)^{-\frac{1}{3}}$. First let us mention that by Lemma~\ref{lm:choice_of_q_possible} below we can choose an ideal $\mathfrak{q}$ which satisfies the conditions needed in order to apply Proposition~\ref{prop:bulk_arb_nf}. We conclude that
\begin{equation}
	\abs{\phi(g)}^2 \ll \abs{T}_{\infty}^{\epsilon}\No(\n)^{\epsilon} \No(\n_0\mathfrak{m}_1)\bigg( \abs{T}_{\infty}^{\frac{5}{6}}\No(\n_2)^{\frac{2}{3}} + \abs{T}_{\R}^{\frac{1}{2}}\abs{T}_{\C}\No(\n_2)^{\frac{1}{2}} \bigg). \nonumber
\end{equation}
\end{proof} 

Next we consider fields that are not totally real. Therefore, we find a maximal, totally real subfield $F^{\R}$. Put $m=[F:F^{\R}]\geq 2$.

\begin{proof}[Proof of Theorem~\ref{th:main_th_2}]
We start by choosing $\mathfrak{q}$ accordingly and using Corollary~\ref{cor:reduction_to_f_dom} to reduce the problem as far as we can. Now observe that for $\abs{y}_{\infty}>\abs{T}_{\infty}^{\frac{1}{4}}$ the estimate in Proposition~\ref{pr:whit_exp_est_bad} gives the upper bound $\No(\n)^{\epsilon}\No(\n_0\mathfrak{m}_1)^{\frac{1}{2}}\abs{T}_{\infty}^{\frac{3}{8}+\epsilon}$. Therefore, by using Proposition~\ref{prop:bulk_bound_cmpF}, we obtain the uniform bound
\begin{equation}
	\frac{\sigma(v^{\circ})(g)}{\Vert \sigma(v^{\circ})\Vert_2} \ll_{F,\epsilon}(\No(\n_2)\No(\n_0\mathfrak{m}_1)\abs{T}_{\infty})^{\frac{1}{2}+\epsilon}\left( \abs{T}_{\infty}^{-\frac{1}{8m-8}}+(\abs{T}_{\infty}\No(\n_2))^{-\frac{1}{8m-4}} \right). \nonumber
\end{equation}
But if $\abs{T}_{\infty}^{-\frac{1}{8m-8}} \geq \No(\n_2)^{-\frac{1}{4}}$ we can use Theorem~\ref{th:main_th_1} to get a better bound. This leads to
\begin{equation}
	\frac{\sigma(v^{\circ})(g)}{\Vert \sigma(v^{\circ})\Vert_2} \ll_{F,\epsilon}(\No(\n_2)\No(\n_0\mathfrak{m}_1)\abs{T}_{\infty})^{\frac{1}{2}+\epsilon}\left( \min(\abs{T}_{\infty}^{-\frac{1}{8m-8}},\No(\n_2)^{-\frac{1}{4}})+(\abs{T}_{\infty}\No(\n_2))^{-\frac{1}{8m-4}} \right). \nonumber
\end{equation}
Now one concludes by interpolation as in \cite{BHMM16}.
\end{proof}

\appendix

\section{Evaluation of some integrals}

In this appendix we will evaluate local integrals that appear in the residual part ot the spectral expansion. More precisely we will calculate the integral
\begin{equation}
	I_{\p}(\chi_{\p})= \int_{Z(F_{\p})\setminus G(F_{\p})} f_{\p}(g)\chi_{\p}(\det(g))dg  \label{eq:definition_of_I_chi} 
\end{equation}
for all $f_{\p}$ defined in Section~\ref{sec:spec_exp_amp}.

First we consider 
\begin{equation}
	f_{\p}(g) = \kappa_{\p,k}(g) = \begin{cases} \omega_{\pi}(z)^{-1} & \text{ for } g=z\ \in Z(F)X_{\p,k}, \\ 0 &\text{ else,} \end{cases} \nonumber
\end{equation}
for some $k$.

\begin{lemma} \label{lm:unram_resid_integral}
For $k\geq 0$ we have
\begin{equation}
	\int_{Z(F_{\p})\setminus G(F_{\p})} \kappa_{\p,k}(g) \chi_{\p}(\det(g))dg = \begin{cases} \chi_{\p}(\varpi_{\p}^k)\text{vol}(X_{\p,k}) &\text{ if $\chi_{\p}$ is unramified,} \\ 0 &\text{ else.} \end{cases} \nonumber
\end{equation}
\end{lemma}
\begin{proof}
The calculation for unramified $\chi_{\p}$ is straight forward. Now we assume that $\chi_{\p}$ is ramified. In this case let us write $X_{\p,k} = \sqcup_i \alpha_i K_{\p}$. Then we clearly have
\begin{equation}
	\int_{Z(F_{\p})\setminus G(F_{\p})}  \kappa_{\p,k}(g) \chi_{\p}(\det(g))dg = \sum_i \chi_{\p}(\det(\alpha_i)) \int_{Z(F_{\p})\setminus G(F_{\p})} \chi_{\p}(\det(g)) \mathbbm{1}_{K_{\p}}(g) dg. \nonumber
\end{equation}
We now use our choice of Haar measure and the fact $\mathbbm{1}_{K_{\p}}(n(x)a(y)k) = \mathbbm{1}_{\op_{\p}}(x) \mathbbm{1}_{\op_{\p}^{\times}}(y)$ to obtain
\begin{eqnarray}
	&&\int_{Z(F_{\p})\setminus G(F_{\p})}  \kappa_{\p,k}(g) \chi_{\p}(\det(g))dg  \nonumber \\
	&&\qquad\qquad\qquad= \sum_i \chi_{\p}(\det(\alpha_i)) \int_{\op_{\p}} \int_{K_{\p}} \kappa_{\p,k}(k) \int_{\op_{\p}^{\times}} \chi_{\p}(y) d\mu^{\times}(y) d\mu_{K_{\p}}(k) d\mu(x) = 0.  \nonumber
\end{eqnarray}
This concludes the proof.
\end{proof}

Next we look at 
\begin{equation}
	f_{\p}(g) = \begin{cases} \text{vol}(Z(\op_{\p})\setminus \tilde{K}_{0,\p}(1))^{-1} \omega_{\pi}^{-1}(z) &\text{ if } g = zk\in Z(F_{\p})\tilde{K}_{0,\p}(1), \\ 0 &\text{ else.} \end{cases} \nonumber
\end{equation} 

\begin{lemma} \label{lm:q_resid_integral}
For a quadratic character $\chi_{\p}$ and unramified $\omega_{\pi}$ we have
\begin{equation}
	I_{\p}(\chi_{\p}) = \begin{cases} 1 &\text{ if } a(\chi_{\p}) \leq 1, \\ 0 &\text{ else.} \end{cases} \nonumber
\end{equation}
\end{lemma}
\begin{proof}
We first observe that for each $g\in \tilde{K}_{0,\p}(1)$ we have $\det(g) \in (\op_{\p}^{\times})^2 + \varpi_{\p} \op_{\p}$. Thus, if $a(\chi_{\p}) \leq 1$ we have $\chi_{\p}(g) = 1$ for all $g\in \tilde{K}_{0,\p}(1)$.

Let us now assume $a(\chi_{\p}) =b >1$. Since $\chi_{\p} \circ \det$ is trivial on 
\begin{equation}
	K_{1,\p}^1(b) = \left( \begin{matrix} 1+\varpi_{\p}^b \op_{\p} & \op_{\p} \\ \varpi_{\p}^b \op_{\p} & 1+\varpi_{\p}^b op_{\p} \end{matrix} \right) \cap K_{\p} \nonumber
\end{equation}
we will start by writing down explicit representatives for $\tilde{K}_{0,\p}(1)/K_{1,\p}^1(b)$. We obtain
\begin{equation}
	\tilde{K}_{0\,p}(1) = \bigsqcup_{\substack{a\in \op_{\p}^{\times}/(1+\varpi_{\p}\op_{\p}), \\ b,a',d' \in \op_{\p}/\varpi_{\p}^{b-1}\op_{\p} }} \left(  \begin{matrix} a+\varpi a'  & 0 \\ \varpi b & a+ \varpi d'  \end{matrix}\right)K_{1,\p}^1(b). \nonumber
\end{equation}
Therefore, we have
\begin{equation}
	I_{\p}(\chi_{\p}) = \sum_{\substack{a\in \op_{\p}^{\times}/(1+\varpi_{\p}\op_{\p}), \\ b,a',d' \in \op_{\p}/\varpi_{\p}^{b-1}\op_{\p} }} \chi_{\p}(a+\varpi_{\p} a')\chi_{\p}(a+\varpi_{\p} d') = 0. \nonumber
\end{equation}
\end{proof}

After this warm up we come to the most interesting case. We consider the truncated matrix coefficient. More precisely we look at
\begin{equation}
	f_{\p}(g)=\Phi_{\pi_{\p}'}'(g) = \begin{cases}
		\Phi_{\pi_{\p}'}(a(\varpi_{\p}^{-n_{1,\p}})ga(\varpi_{\p}^{-n_{11,\p}})) &\text{ if } g\in ZK_{\p}^0,\\
		0&\text{ else,}
	\end{cases} \nonumber
\end{equation}
with
\begin{equation}
	K_{\p}^0= \begin{cases}
		K_{\p} &\text{ if $n$ is even,} \\
		K_{\p}^0(1) &\text{ if $n$ is odd}.
	\end{cases} \nonumber
\end{equation}

We can now compute the following.

\begin{lemma} \label{lm:ram_resid_integral}
If $\chi_{\p}^2 = \omega_{\pi_{\p}}$ one has
\begin{equation}
	I_{\p}(\chi_{\p})  =  0  \nonumber
\end{equation}
Unless $a(\pi_{\p})=1$. In this case the integral may be non-zero but we still have $I_{\p}(\chi_{\p})\geq 0$.
\end{lemma}

Before we begin with the proof we recall some properties of $\Phi_{\pi_{\p}'}$. For any unitary, generic representation $\pi_{\p}'$ of $G(F)$ we define the matrix coefficient associated to a Whittaker new vector $W_{\pi_{\p}'}$ by
\begin{equation}
	\Phi_{\pi_{\p}'}(g) = \SP{W_{\pi_{\p}'}}{\pi_{\p}'(g)W_{\pi_{\p}'}}. \nonumber
\end{equation}
Using the finite Fourier coefficients $c_{t,l,}(\mu)$ defined in \cite[(1.5)]{As17_1} we can prove the following nice formula.

\begin{lemma} \label{lm:mat_coeff_via_ctlv}
We have
\begin{equation}
	\Phi_{\pi_{\p}'}(n(x) g_{t,l,1})= \sum_{m\in\Z} W_{\pi_{\p}'}(a(\varpi_{\p}^m))\sum_{\mu\in \mathfrak{X}_l} \overline{c_{t+m,l}(\mu)}G(-\varpi_{\p}^m x, \omega_{\pi_{\p}'}\mu). \nonumber
\end{equation}
\end{lemma}
\begin{proof}
First we use the definition of $\Phi_{\pi}$. We arrive at
\begin{eqnarray}
	\Phi_{\pi_{\p}'}(n(x)g_{t,l,1}) &=& \SP{W_{\pi_{\p}'}}{\pi_{\p}'(n(x)g_{t,l,1})W_{\pi_{\p}'}} \nonumber \\
	&=& \int_{F_{\p}^{\times}} W_{\pi_{\p}'}(a(y))\overline{W_{\pi_{\p}'}(a(y)n(x)g_{t,l,1})}d\mu^{\times}(y) \nonumber \\
	&=& \sum_{m\in \Z} W_{\pi_{\p}'}(a(\varpi_{\p}^m))\int_{\op_{\p}^{\times}}\omega_{\pi_{\p}'}(v) \overline{W_{\pi_{\p}'}(a(\varpi_{\p}^mv)n(x)g_{t,l,1})}d\mu^{\times}(v). \nonumber
\end{eqnarray}
It is straight forward to check that 
\begin{equation}
	a(\varpi_{\p}^m v)n(x)g_{t,l,1} = n(\varpi_{\p}^m vx)g_{t+m,l,v^{-1}}\underbrace{\left(\begin{matrix} 1&0\\0&v \end{matrix} \right)}_{\in K_{1,\p}(n_{\p})}. \nonumber
\end{equation}
Inserting this together with \cite[(11)]{Sa15_2} and the definition of the Gauss sum completes the proof.
\end{proof}

\begin{proof}[Proof of Lemma~\ref{lm:ram_resid_integral}]
Put $b=\max(a(\chi_{\p}),n_{\p})$. Then $\chi_{\p}\circ \det$ and $\Phi_{\pi_{\p}}$ are bi-$K_{1,\p}(b)$-invariant. Further we recall 
\begin{equation}
	\Phi_{\pi_{\p}'}'(g) = \mathbbm{1}_{ZK_{\p}^{\circ}}(g) \Phi_{\pi_{\p}'}(a(\varpi_{\p}^{-n_{1,\p}})ga(\varpi_{\p}^{n_{1,\p}})). \nonumber
\end{equation}
Thus a simple change of variables yields
\begin{equation}
	I_{\p}(\chi_{\p})=\int_{Z(F_{\p})\setminus G(F_{\p})} \chi_{\p}(\det(g)) \Phi_{\pi_{\p}'}(g) \mathbbm{1}_{ZK_{\p}^{\circ}}(a(\varpi_{\p}^{n_{1,\p}})ga(\varpi_{\p}^{-n_{1,\p}}))dg. \nonumber
\end{equation}
It is easy to check, that $\mathbbm{1}_{ZK_{\p}^{\circ}}(a(\varpi_{\p}^{n_{1,\p}})\cdot a(\varpi_{\p}^{-n_{1,\p}}))$ is bi-$K_{0,\p}(b)$-invariant. And therefore, the hole integrand is bi-$K_{0,\p}(b)$-invariant, so that we can use \cite[Lemma~3.2.4]{Co17}. This yields
\begin{equation}
	I_{\p}(\chi_{\p}) = \sum_{l=0}^b c_l \sum_{t\in \Z} q_{\p}^{t+l} \int_{F_{\p}} \chi_{\p}(\varpi_{\p}^t) \Phi_{\pi_{\p}'}(n(x)g_{t,l,1}) \mathbbm{1}_{ZK_{\p}^{\circ}}(a(\varpi_{\p}^{n_{1,\p}})n(x)g_{t,l,1}a(\varpi_{\p}^{-n_{1,\p}})) d\mu_{\p}(x), \nonumber
\end{equation}
for some positive constants $c_l$. We remark that since $\omega_{\pi_{\p}'}$ is trivial on the uniformizer, so is $\chi_{\p}$. Next we will investigate which restrictions on $x$, $l$, and $t$ are imposed by the characteristic function (up to the centre). One checks
\begin{equation}
	a(\varpi_{\p}^{n_{1,\p}})n(x)g_{t,l,1}a(\varpi_{\p}^{-n_{1,\p}}) = z\cdot\left( \begin{matrix} \varpi_{\p}^k x & \varpi_{\p}^{n_{1,\p}-l+k}x-\varpi_{\p}^{t+n_{1,\p}+k} \\ \varpi_{\p}^{-n_{1,\p}+k} & \varpi_{\p}^{k-l} \end{matrix}\right). \nonumber
\end{equation}
Here we use the center to force all coefficients to be in $\op_{\p}$. This holds for 
\begin{equation}
	k\geq \max(n_{1,\p},l,-v_{\p}(x),-v_{\p}(x)-n_{1,\p}+l) \nonumber
\end{equation}
and suitable $t$. But we also need to make sure that the determinant is in $\op_{\p}^{\times}$. This implies $t+2k = 0$.

We now consider $n_{\p}$ to be even. In this case $K_{\p}^{\circ} = K_{\p}$ and  we get the conditions
\begin{equation}
	k = n_{1,\p}, t= -2n_{1,\p}, l \leq n_{1,\p},\text{ and } -v_{\p}(x) \leq n_{1,\p}. \label{eq:cond_for_even}
\end{equation}
After inserting the formula from Lemma~\ref{lm:mat_coeff_via_ctlv} for the matrix coefficient we obtain
\begin{equation}
	I_{\p}(\chi_{\p}) = \sum_{l=0}^{n_{1,\p}}c_lq_{\p}^{l-2n_{1,\p}}\sum_{m\in\Z}W_{\pi_{\p}'}(a(\varpi_{\p}^m))\sum_{\mu\in \mathfrak{X}_l}\overline{c_{t+m}(\mu)}\int_{\varpi_{\p}^{-n_{1}}\op_{\p}}G(-\varpi_{\p}^m x,\omega_{\pi_{\p}'}\mu)d\mu(x).  \label{eq:exp_sum_even}
\end{equation}
Inserting the evaluation of the Gauss sum given in \cite{Sa15_2} together with character orthogonality shows that most of the integrals vanish. We are left with
\begin{eqnarray}
	I_{\p}(\chi_{\p}) &=& \sum_{l=a(\omega_{\pi_{\p}'})}^{n_{1,\p}} c_l q^{l-2n_{1}} \sum_{m\in \Z} W_{\pi_{\p}'}(a(\varpi_{\p}^m)) \overline{c_{m-2n_{1,\p},l}(\omega_{\pi_{\p}'}^{-1})}  \sum_{t\geq 0} q_{\p}^{-t+n_{1,\p}}\int_{\op_{\p}^{\times}} G(-\varpi_{\p}^{m+t-n_{1}}x, 1) d\mu(x) \nonumber 
\end{eqnarray}
Now we have to consider different cases. First we deal with representations that satisfy $L(s,\pi_{\p}) = 1$. In this case using \cite[(6),(7)]{Sa15_2} yields
\begin{eqnarray}
	I_{\p}(\chi_{\p}) &=&  \sum_{l=a(\omega_{\pi_{\p}'})}^{n_{1,\p}} c_l q_{\p}^{l-2n_{1,\p}}  \overline{c_{-2n_{1},l}(\omega_{\pi_{\p}'}^{-1})}   \sum_{t\geq 0} q_{\p}^{n_{1,\p}-t} \int_{\op_{\p}^{\times}} G(\varpi_{\p}^{t-n_{1,\p}}v,1) d\mu(v) \nonumber \\
	&=&  \sum_{l=a(\omega_{\pi_{\p}'})}^{n_{1,\p}} c_l q_{\p}^{l-2n_{1,\p}}  \overline{c_{-2n_{1},l}(\omega_{\pi_{\p}'}^{-1})}  \left[\sum_{t\geq n_{1,\p}} q_{\p}^{n_{1}-t} \zeta_{F_{\p}}(1)-1\right] = 0. \nonumber
\end{eqnarray}
Next we consider the case $\pi_{\p}' = \chi_1\boxplus \chi_2$ with $a(\chi_1) > a(\chi_2)=0$. In this case we have $a(\omega_{\pi_{\p}'}) = a(\chi_1) = n_{\p} >0$. Recall that at the moment we are considering $n_{\p}$ even. Thus, $a(\omega_{\pi_{\p}'}) > n_{1,\p} \geq 1$. We conclude that $I_{\p}(\chi_{\p}) = 0$ since the $l$-sum is empty.
Let us remark, that $\pi_{\p}' = \chi St$ for unramified $\chi$ has conductor 1 and therefore does not need to be considered yet.

We have checked that $I_{\p}(\chi_{\p}) = 0$ for even $n_{\p}$ by considering all necessary types of $\pi_{\p}'$. Now let us move on to $n_{\p}$ odd. In this case $K_{\p}^{\circ} = K_{\p}^0(1)$ and additionally to \eqref{eq:cond_for_even} the characteristic function also forces $v_{\p}(\varpi_{\p}
^{2n_{1,\p}-l}x-1) \geq 1$. This implies
\begin{equation}
	l=n_{1,\p} \text{ and } x\in \varpi_{\p}^{-n_{1,\p}}(1+\varpi_{\p}\op_{\p}) \nonumber
\end{equation}
Analogously to \eqref{eq:exp_sum_even} we get
\begin{eqnarray}
	I_{\p}(\chi_{\p}) &=& c_{n_{1,\p}} q_{\p}^{-n_{1,\p}} \sum_{m\in\Z} W_{\pi_{\p}'}(a(\varpi_{\p}^m))\sum_{\mu\in \mathfrak{X}_{n_{1,\p}}} \overline{c_{-2n_{1,\p}+m,n_{1,\p}}(\mu)}\int_{\varpi_{\p}^{-n_{1,\p}}(1+\varpi_{\p}\op_{\p})}G(-\varpi_{\p}^m x,\omega_{\pi_{\p}'}\mu) d\mu(x) \nonumber \\
	&=&  c_{n_{1,\p}} \sum_{m\in\Z} W_{\pi_{\p}}(a(\varpi_{\p}^m)) \sum_{\substack{\mu \in \mathfrak{X}_{n_{1}}, \\ a(\mu\omega_{\pi_{\p}'}) \leq 1}} \overline{c_{-2n_{1,\p}+m,n_{1,\p}}(\mu)}\int_{(1+\varpi_{\p}\op_{\p})}G(-\varpi_{\p}^{m-n_{1}}x,\omega_{\pi_{\p}'}\mu) d\mu(x) \nonumber
\end{eqnarray}
In the last step we used again the Gauss sum evaluation \cite[(6)]{Sa15_2} and orthogonality of characters to remove all $\mu$ with $a(\mu\omega_{\pi_{\p}'})>1$.

Now we have to consider different cases again. First let us look at $\pi_{\p}$ with $L(s,\pi_{\p}) = 1$. In this case me have $n_{\p} >2$, since we assume $n_{\p}$ odd. By \cite[(7)]{Sa15_2}  we get
\begin{equation}
	I_{\p}(\chi_{\p}) = c_{n_{1,\p}} \sum_{\substack{\mu \in \mathfrak{X}_{n_{1}}, \\ a(\mu\omega_{\pi_{\p}'}) \leq 1}} \overline{c_{-2n_{1,\p},n_{1,\p}}(\mu)}\int_{(1+\varpi_{\p}\op_{\p})}\underbrace{G(-\varpi_{\p}^{-n_{1,\p}}x,\omega_{\pi_{\p}'}\mu)}_{=0} d\mu(x) = 0 \nonumber
\end{equation}

Next let $\pi_{\p} =\chi_1\boxplus\chi_2$ with $a(\chi_1) > a(\chi_2) = 0$. If $n_{\p} = a(\chi_1) > 1$ we immediately have $a(\omega_{\pi_{\p}'} \mu)> n_{1,\p}$ for all $\mu\in \mathfrak{X}_{n_{1,\p}}$. Thus, in these cases $I_{\p}(\chi_{\p}) = 0$. So we can assume $1=n_{\p} = n_{1,\p} = a(\chi_1)$. Using \cite[(6),(7)]{Sa15_2} we have the identity
\begin{eqnarray}
	I_{\p}(\chi_{\p}) = c_1 \vol(1+\varpi_{\p}\op_{\p}, \mu) \bigg( && \sum_{\substack{\mu\in \mathfrak{X}_1, \\ \mu\neq \omega_{\pi_{\p}'}^{-1}} } \overline{c_{-2,1}(\mu)}\zeta_{F_{\p}}(1) q_{\p}^{-\frac{1}{2}}\epsilon(\frac{1}{2},\omega_{\pi_{\p}'}^{-1}\mu^{-1})\omega_{\pi_{\p}'}(-1)\mu(-1)  \nonumber\\
	&& +\sum_{m\geq 1} \chi_1(\varpi_{\p}^m)q_{\p}^{-\frac{m}{2}} \overline{c_{-2+m,1}(\omega_{\pi_{\p}'}^{-1})}  \nonumber\\
	&& - \zeta_{F_{\p}}(1) q_{\p}^{-1}\overline{c_{-2,1}(\omega_{\pi_{\p}'}^{-1})}  \bigg). \nonumber 
\end{eqnarray}
Inserting the expressions for $c_{t,1}(\cdot)$ given in \cite[Lemma~2.3]{As17_1} yields
\begin{eqnarray}
	I_{\p}(\chi_{\p}) &=& c_1 \vol(1+\varpi_{\p}\op_{\p}, d\mu) \omega_{\pi_{\p}'}(-1) \left( \sum_{\mu\neq \omega_{\pi_{\p}'}^{-1}} \zeta_{F_{\p}}(1)^2q_{\p}^{-1} + \sum_{m\geq 1}q_{\p}^{-m}+\zeta_{F_{\p}}(1)^2 q_{\p}^{-2} \right) \nonumber \\
	&=& c_1 \vol(1+\varpi_{\p}\op_{\p}, d\mu) \omega_{\pi_{\p}'}(-1) \left( \zeta_{F_{\p}}(1)^2q_{\p}^{-1}(q_{\p}-2)+\zeta_{F_{\p}}(1)q_{\p}^{-1}+\zeta_{F_{\p}}(1)^2q_{\p}^{-2} \right). \nonumber
\end{eqnarray}
Now we observe $\omega_{\pi_{\p}'}(-1) = \chi_{\p}(-1)^{2} = 1$ and deduce $I_{\p}(\chi_{\p}) \geq 0$.

This leaves us with the last case $\pi_{\p}' = \chi St$ for unramified $\chi$. Note that in this case $\omega_{\pi} = \chi_{\p}^2 = 1$ since we assumed $\omega_{\pi_{\p}'}(\varpi)=1$. Thus we are dealing with $\pi_{\p} = St$ and we have $a(\pi_{\p}') = n_{\p} = n_{1,\p} = 1$. We obtain
\begin{equation}
	I_{\p}(\chi_{\p}) = c_1 \sum_{m\geq 0} q_{\p}^{-m} \sum_{\mu\in \mathfrak{X}_1}\overline{c_{m-2,1}(\mu)}\int_{1+\varpi_{\p}\op_{\p}}G(-\varpi_{\p}^{m-1}x,\mu)d\mu(x) . \nonumber
\end{equation}
Evaluating the Gauss sum reveals
\begin{eqnarray}
	I_{\p}(\chi_{\p}) = c_1 \vol(1+\varpi_{\p}\op_{\p}) \bigg( && \sum_{a(\mu) = 1} \zeta_{F_{\p}}(1) q_{\p}^{-\frac{1}{2}}\epsilon(\frac{1}{2},\mu^{-1}) \mu(-1) \overline{c_{-2,1}(\mu)} \nonumber\\
	&& +\sum_{m\geq 1} q_{\p}^{-m} \overline{c_{m-2,1}(1)} \nonumber \\
	&& -\zeta_{F_{\p}}(1)q_{\p}^{-1}\overline{c_{-2,1}(1)}\bigg).\nonumber
\end{eqnarray}
Using the evaluations oc $c_{t.l}(\cdot)$ given in \cite[Lemma~2.1]{As17_1} this simplifies to
\begin{eqnarray}
	I_{\p}(\chi_{\p}) &=& c_1\vol(1+\varpi_{\p}\op_{\p}) \left( \sum_{a(\mu)=1}\zeta_{F_{\p}}(1)^2q_{\p}^{-1} + \sum_{m\geq 1} q_{\p}^{-2m}+\zeta_{F_{\p}}(1)^2q_{\p}^{-2} \right) \nonumber  \\
	&=& c_1\vol(1+\varpi_{\p}\op_{\p}) \left( \zeta_{F_{\p}}(1)^2 q_{\p}^{-1}(q_{\p}-2) + q_{\p}^{-2}\zeta_{F_{\p}}(2)+ \zeta_{F_{\p}}(1)^2q_{\p}^{-2}\right)> 0 \nonumber
\end{eqnarray}
This was the last case to consider and concludes the proof.
\end{proof}

\bibliographystyle{plain}
\bibliography{bibliography} 

\begin{thebibliography}{10}

\bibitem{As17_1}
E.~Assing.
\newblock On the size of p-adic {W}hittaker functions.
\newblock {\em arXiv}, 2017.

\bibitem{BHMM16}
V.~Blomer, G.~Harcos, P.~Maga, and D.~Mili{\'c}evi{\'c}.
\newblock The sup-norm problem for $\operatorname{GL}(2)$ over number fields.
\newblock {\em arXiv preprint arXiv:1605.09360}, 2016.

\bibitem{BJ79}
A.~Borel and H.~Jacquet.
\newblock Automorphic forms and automorphic representations.
\newblock In {\em Automorphic forms, representations and L-functions (Proc.
  Sympos. Pure Math., Oregon State Univ., Corvallis, Ore., 1977), Part},
  volume~1, pages 189--207, 1979.

\bibitem{BL05}
A.~Borel and L.~Ji.
\newblock Compactifications of symmetric and locally symmetric spaces.
\newblock {\em Lie Theory}, pages 69--137, 2005.

\bibitem{MB16}
F.~Brumley and S.~Marshall.
\newblock Lower bounds for {M}aass forms on semisimple groups.
\newblock {\em arXiv preprint arXiv:1604.02019}, 2016.

\bibitem{TB14}
F.~Brumley and N.~Templier.
\newblock Large values of cusp forms on $\operatorname{GL}(n)$.
\newblock {\em arXiv preprint arXiv:1411.4317}, 2014.

\bibitem{Bu96}
D.~Bump.
\newblock {\em Automorphic forms and representations}.
\newblock Cambridge University Press, 1996.

\bibitem{Co17}
A.~Corbett.
\newblock Period integrals and {$\operatorname{L}$}-functions in the theory of
  automorphic forms.
\newblock {\em Phd thesis}, 2017.

\bibitem{Fl79}
D.~Flath.
\newblock Decomposition of representations into tensor products.
\newblock In {\em Automorphic forms, representations and L-functions (Proc.
  Sympos. Pure Math., Oregon State Univ., Corvallis, Ore., 1977), Part},
  volume~1, pages 179--183, 1979.

\bibitem{GJ79}
S.~Gelbart and H.~Jacquet.
\newblock Forms of {$\operatorname{GL}(2)$} from the analytic point of view.
\newblock In {\em Proc. Sympos. Pure Math}, volume~33, pages 213--251, 1979.

\bibitem{Hi80}
J.~Hinz.
\newblock Eine {E}rweiterung des nullstellenfreien {B}ereiches der {H}eckeschen
  {Z}etafunktion und {P}rimideale in {I}dealklassen.
\newblock {\em Acta Arithmetica}, 38(3):209--254, 1980.

\bibitem{HL94}
J~Hoffstein and P.~Lockhart.
\newblock Coefficients of {M}aas forms and the {S}iegel zero.
\newblock {\em Annals of Mathematics}, 1994.

\bibitem{La13}
S.~Lang.
\newblock {\em Algebraic number theory}, volume 110.
\newblock Springer Science \& Business Media, 2013.

\bibitem{Ne13}
J.~Neukirch.
\newblock {\em Algebraic number theory}, volume 322.
\newblock Springer Science \& Business Media, 2013.

\bibitem{Sa15}
A.~Saha.
\newblock Hybrid sup-norm bounds for {M}aass newforms of powerful level.
\newblock {\em arXiv preprint arXiv:1509.07489}, 2015.

\bibitem{Sa15_2}
A.~Saha.
\newblock Large values of newforms on {$\operatorname{GL}(2)$} with highly
  ramified central character.
\newblock {\em International Mathematics Research Notices},
  2016(13):4103--4131, 2015.

\bibitem{Te15}
N.~Templier.
\newblock Hybrid sup-norm bounds for {H}ecke--{M}aass cusp forms.
\newblock {\em Journal of the European Mathematical Society}, 17(8):2069--2082,
  2015.

\bibitem{Za15}
A.~Zaman.
\newblock Bounding the least prime ideal in the {C}hebotarev density theorem.
\newblock {\em arXiv preprint arXiv:1508.00287}, 2015.

\end{thebibliography}

\end{document}